\documentclass{amsart}

\title[The elementary theory of FnBTs]{The elementary theory of full $n$-branching ordinal trees with successor functions}

\author[Chen]{Junhong Chen}
 \address[Junhong Chen]
         {School of Mathematical Science, Fudan University, 220 Handan Road, Shanghai, 200433 China}
\email{21300180086@m.fudan.edu.cn}

\author[Zhang]{Yi Zhang}
 \address[Yi Zhang]
         {College of Intelligent Robotics and Advanced Manufacturing, Fudan University, 220 Handan Road, Shanghai, 200433 China}
 \email{pacisaury@gmail.com}

\usepackage{latexsym,amsfonts,amsmath,amssymb,mathrsfs,tikz-cd,adjustbox}
\usetikzlibrary{fit,decorations.pathmorphing,backgrounds}
\usepackage{changepage,calc}
\usepackage[utf8]{inputenc}
\usepackage{csquotes}
\usepackage{braket}
\usepackage{bbm}
\usepackage[hidelinks]{hyperref}
\usepackage{relsize}
\usepackage[cmyk,svgnames,dvipsnames]{xcolor}
\usepackage{tikz}
\usetikzlibrary {3d}
\usepackage{comment}
\usetikzlibrary{arrows,arrows.meta,petri,topaths,positioning,shapes,shapes.misc,patterns,calc,decorations.pathreplacing,hobby}
\pgfdeclarelayer{foreground}
\pgfdeclarelayer{background}
\pgfdeclarelayer{boardarrows}
\pgfdeclarelayer{boardgrid}
\pgfdeclarelayer{boardshades}
\pgfsetlayers{boardshades,bo
ardgrid,boardarrows,background,main,foreground}
\usepackage{wrapfig} 
\usepackage{caption} 
\usepackage{float}
\RequirePackage{doi}

\usepackage{enumitem}

\usepackage[backend=biber,style=alphabetic,maxbibnames=15,maxcitenames=6,dateabbrev=false,safeinputenc=true]{biblatex}

\renewcommand{\UrlFont}{} 
\renewbibmacro{in:}{\ifentrytype{article}{}{\printtext{\bibstring{in}\intitlepunct}}} 
\DeclareFieldFormat{url}{\UrlFont\url{#1}} 
\DeclareFieldFormat{urldate}{
  (version \thefield{urlday}\addspace%
  \mkbibmonth{\thefield{urlmonth}}\addspace%
  \thefield{urlyear}\isdot)}
\DeclareFieldFormat{eprint:arxiv}{
  \ifhyperref
    {\href{http://arxiv.org/abs/#1}{%
        arXiv\addcolon\nolinkurl{#1}}\iffieldundef{eprintclass}{}{\UrlFont{\mkbibbrackets{\thefield{eprintclass}}}}}
    {arXiv\addcolon\nolinkurl{#1}\iffieldundef{eprintclass}{}{\UrlFont{\mkbibbrackets{\thefield{eprintclass}}}}}}

\newcommand{\cf}{\operatorname{cf}}

\newtheorem{theorem}{Theorem}
\newtheorem{lemma}{Lemma}
\newtheorem{corollary}{Corollary}
\theoremstyle{definition}

\begin{document}

\begin{abstract}
    We investigate the first-order theories of full \(n\)-branching ordinal trees \(\mathfrak{T}_\alpha^n\). We obtain a complete classification of the standard trees up to elementary equivalence: every \(\mathfrak{T}_\alpha^n\) is elementarily equivalent to one of four canonical types determined by the ordinal \(\alpha\). For each canonical type we establish effective quantifier elimination and prove decidability of the theory. Along the way we develop the realizability theory of colored ordinal characters and obtain a sharp bound on the length of minimal witnesses. We also clarify the relationship between these trees and monadic second-order logic over ordinals, and show that the equal-height relation is not first-order definable in any standard tree.
\end{abstract}

\maketitle
\pagestyle{plain}

\tableofcontents

\newcommand\YZ[1]{{\color{red}(YZ: #1)}}
\newcommand\CJ[1]{{\color{green}(CJ: #1)}}

\section{Introduction}\label{sec:introduction}

Trees endowed with additional structure arise naturally in many parts of logic and theoretical computer science. A particularly well-behaved class is formed by the full $n$-branching ordinal trees: for a fixed finite $n\geq 2$ and a limit ordinal $\alpha$, consider the set $n^{<\alpha}$ of all sequences of length less than $\alpha$ with entries in $n=\{0,1,\dots,n-1\}$, ordered by initial segment or, equivalently, by set-theoretic inclusion. These sequences represent the nodes of a rooted tree in which every node has exactly $n$ immediate successors, labeled by the \emph{successor functions} $S_0,\dots,S_{n-1}$, and every maximal chain (path) is isomorphic to the ordinal $\alpha$. We denote the resulting first-order structure by
$$\mathfrak{T}_\alpha^n=(n^{<\alpha},\subseteq,S_0,\dots,S_{n-1}).$$

These structures will be abbreviated as FnBTs. The first-order theories of trees have been studied extensively (see, e.g., \cite{Kellerman2021,KellermanGoranko2021} and references therein), but the combination of full branching and a fixed ordinal height introduces a rich interplay between the combinatorial complexity of the tree, automata theory, and the model theory of linear orders that is far from being fully understood.

In this paper, we undertake a systematic investigation of the structures $\mathfrak{T}_\alpha^n$ and their elementary theories. Our work reveals a series of surprisingly clean classification and decidability results and delineates the precise expressive power of the first-order language of these trees.

\begin{theorem}
    Suppose that $\alpha$ is a nonzero limit ordinal. Exactly one of the following four possibilities holds:
    \begin{enumerate}
        \item If $\alpha\equiv\omega^\omega$, then $\mathfrak T_\alpha^n\equiv\mathfrak T_{\omega^\omega}^n$.
        \item If $\alpha=\theta<\omega^\omega$, then $\mathfrak T_\alpha^n=\mathfrak T_\theta^n$.
        \item If $\alpha=\alpha_0+\theta$, where $\alpha_0\equiv\omega^\omega$, $\cf(\alpha_0)=\omega$, and $0<\theta<\omega^\omega$, then $\mathfrak T_\alpha^n\equiv\mathfrak T_{\omega^\omega+\theta}^n$.
        \item If $\alpha=\alpha_0+\theta$, where $\alpha_0\equiv\omega^\omega$, $\cf(\alpha_0)>\omega$, and $0<\theta<\omega^\omega$, then $\mathfrak T_\alpha^n\equiv\mathfrak T_{\omega_1+\theta}^n$.
    \end{enumerate}
    In every case, $\operatorname{Th}(\mathfrak T_\alpha^n)$ is decidable and admits effective quantifier elimination in an appropriate expanded language. These four kinds of trees are pairwise elementarily inequivalent.
\end{theorem}
\begin{proof}
  These assertions follow from the results below.

  For the decidability assertions, we assume effective notations for the relevant exponents and cofinalities.

  For Type~(1): Theorems~\ref{thm:1-qe} and~\ref{thm:1-classify}.

  For Type~(2): Theorems~\ref{thm:2-qe} and~\ref{thm:2-classify}.

  For Types~(3) and~(4): Theorems~\ref{thm:3-qe} and~\ref{thm:3-classify},

  Pairwise elementary inequivalence is Theorem~\ref{theorem:inequivalent}.

\end{proof}

The paper is organized as follows. Section~\ref{sec:preliminaries} collects the necessary background on trees, colored ordinals, Ehrenfeucht–Fra\"{\i}ss\'e games, and monadic second-order logic. In Section~\ref{sec:EFbound} we refine the Ehrenfeucht–Fra\"{\i}ss\'e bounds for finitely colored ordinals, establishing both upper and lower bounds of tower-exponential growth and giving an exact computation algorithm. Section~\ref{sec:realizability} solves the realizability problem for FCO characters by proving cofinal and exact stretching theorems, and deduces that realizability is decidable for every ordinal presented in Cantor normal form. Sections~\ref{sec:type1syn}–\ref{sec:type34syn} contain the core of the tree analysis: we introduce four canonical types of full \(n\)-branching trees, axiomatize their first-order theories in appropriate diagram languages, prove quantifier elimination and decidability, and characterize exactly which standard trees are elementarily equivalent to each canonical tree. From these results we extract the common first-order theory of all standard trees. Finally, the last section discusses the relationship between FnBTs and monadic second-order structures on ordinals, establishing mutual interpretability results and proving that the equal-height relation is not first-order definable in any standard tree.

\section{Preliminaries}\label{sec:preliminaries}

We collect in this section the necessary definitions and known facts from the model theory of trees, Ehrenfeucht–Fra\"{\i}ss\'e games, colored linear orders, monadic second order theory of ordinals, and the related automata theory.

Let us start by reviewing some standard notions in logic.\label{subsec:logicprelim}

Fix a finite relational language and a structure $\mathfrak{A}$ in this language. For any finite tuple $\bar a$ in $\mathfrak A$, we define the general $q$-characteristic $[[(\mathfrak{A},\bar{a})]]^q$ recursively as follows:
\begin{enumerate}
    \item $[[(\mathfrak{A},\bar{a})]]^0$ is the conjunction of the atomic and negated atomic formulae satisfied by $\bar a$.
    \item $[[(\mathfrak{A},\bar{a})]]^{q+1}$ records all $q$-characteristics realized by one-point extensions $(\bar a,a_+)$: it contains one formula $\exists x_+[[(\mathfrak{A},\bar{a},a_+)]]^q$ for each realized characteristic and the formula
    $$\forall x_+\bigvee_{a_+}[[(\mathfrak{A},\bar{a},a_+)]]^q,$$
    where the disjunction ranges over representatives of the finitely many realized $q$-characteristics.
\end{enumerate}
There are finitely many such formulae up to logical equivalence for each fixed $q$ and tuple length. Notice that \emph{characteristics} here refer to syntactical objects, and definitions in the following sections of \emph{characters} are prepared as semantical objects.

The following lemma shows that characteristics are useful for concluding elementary equivalences.
\begin{lemma}\label{lem:hintikka}
    $(\mathfrak{A},\bar{a})\equiv_{q}(\mathfrak{B},\bar{b})$ if and only if $[[(\mathfrak{A},\bar{a})]]^q$ and $[[(\mathfrak{B},\bar b)]]^q$ are logically equivalent.
\end{lemma}

Here, the quantifier rank of a formula is the maximal nesting depth of its quantifiers, and $\equiv_q$ means agreement on all formulae of quantifier rank at most $q$. We follow the standard definition of the Ehrenfeucht--Fra\"{\i}ss\'e game: in every round, Spoiler($\forall$) chooses an element in one structure, and Duplicator($\exists$) chooses an element in the other. After $q$ rounds, Duplicator wins when the resulting tuples satisfy the same quantifier-free formulae. Two structures are $q$-elementarily equivalent exactly when Duplicator has a winning strategy in the $q$-round game. We will denote these games as EF games.

Then, there are some standard notions in the theory of ordinals and trees.\label{subsec:babyprelim}

The first‑order theory of well‑orderings was completely analysed by Mostowski and Tarski~\cite{Doner1978}. For our purposes the following consequences are sufficient: 

\begin{theorem}\label{theorem:DMT}
    In the language $\mathcal{L}=\{<, \phi_{\theta}\}$. Two ordinals are elementarily equivalent as linear orders if and only if they are equal modulo $\omega^\omega$. More precisely, $\alpha\equiv\beta$ holds exactly when either $\alpha=\beta<\omega^\omega$, or both $\alpha,\beta\geq\omega^\omega$ and they have the same Cantor normal form below $\omega^\omega$. Also, every $\theta<\omega^\omega$ is definable using a first-order formula $\phi_\theta(x)$ in the language of ordinals.
\end{theorem}

A \emph{$d$-colored ordinal} is a structure $\mathfrak{A}=(A,<,C_0,\dots,C_{d-1})$ where $(A,<)$ is an ordinal, $d\geq 2$ is a natural number, and the unary predicates $C_i$ form a partition of $A$ (each element satisfies exactly one $C_i$). When the number of colors is clear in context, we may briefly speak of \emph{finitely colored ordinals} (FCOs).

For a $d$-colored ordinal $\mathfrak{A}$ and $a\in A$, let $i$ be the unique color with $\mathfrak{A}\models C_i(a)$. The \emph{FCO $q$-character} of $a$ is the triple
$$\rho_q^{\mathfrak{A}}(a)=\bigl(i,\,[[\mathfrak{A}_{<a}]]^{q},\,[[\mathfrak{A}_{>a}]]^{q}\bigr),$$
where $[[\,\cdot\,]]^q$ denotes the $q$-characteristic sentence introduced as before. The \emph{FCO $q$-character} of $\mathfrak{A}$ is the set
$$\rho_q(\mathfrak{A})=\{\,\rho_q^{\mathfrak{A}}(a):a\in A\,\}.$$

We shall use the following standard facts about FCOs. They generally hold for all finitely colored linear orders, but we will not use any of these structures here.
\begin{lemma}[Character and composition facts]\label{lem:colored-facts}
Let $\mathfrak A,\mathfrak B$ be FCOs.
    \begin{enumerate}
        \item\label{itm:colored-character} $\mathfrak{A}\equiv_{q+1}\mathfrak{B}$ if and only if $\rho_q(\mathfrak A)=\rho_q(\mathfrak B)$.
        \item\label{itm:colored-sum} If $\mathfrak A_i\equiv_q\mathfrak B_i$ for every $i$ in a linear order $I$, then
        $$\sum_{i\in I}\mathfrak A_i\equiv_q\sum_{i\in I}\mathfrak B_i.$$
        \item\label{itm:colored-splitting} Suppose $q\geq 1$, $\mathfrak A\equiv_{q+1}\mathfrak B$, and $\mathfrak A=\mathfrak A_0+\mathfrak A_1$, where both summands are nonempty and $\mathfrak A_1$ has a least point $a$. Then $\mathfrak B=\mathfrak B_0+\mathfrak B_1$ for nonempty orders $\mathfrak B_0,\mathfrak B_1$ satisfying $\mathfrak A_i\equiv_q\mathfrak B_i$, with $\mathfrak{B}_1$ having a least point $b$.
    \end{enumerate}
\end{lemma}
\begin{proof}
    Items~(\ref{itm:colored-character}) and~(\ref{itm:colored-sum}) are the character criterion and the sum theorem from \cite{Mwesigye2018}. For item~(\ref{itm:colored-splitting}), the winning strategy for $\exists$ in the $q+1$-round EF game choose the correct splitting point $b$ once $\forall$ choose $a$ at the beginning.
\end{proof}

A \emph{tree} is a partially ordered set $(T,<)$ with a least element $\epsilon$ (the \emph{root}) such that for every $x\in T$ the set $T_{<x}=\{y\in T:y<x\}$ is well‑ordered by $<$, and any two elements have a common lower bound. To make everything clear, we take the further requirement that every pair $x,y$ admits a greatest common lower bound, denoted by $x\wedge y$.

A maximal totally ordered subset of $T$ is called a \emph{path}. The tree is an \emph{$\alpha$-tree} if every path is order‑isomorphic to the ordinal $\alpha$. A subset $J\subseteq T$ is \emph{convex} if $a,b\in J$ and $a<c<b$ imply $c\in J$. An element $y$ is an \emph{immediate successor} of $x$ if $x<y$ and there is no $z$ with $x<z<y$; similar definition works for \emph{immediate predecessor}. The tree is \emph{full $n$-branching} if every non‑maximal node has exactly $n$ immediate successors.

Full $n$-branching $\alpha$-trees have an especially simple standard form, as we wish. The trees $\mathfrak{T}_\alpha^n$ will be called \emph{standard trees} throughout this paper.
\begin{theorem}[Standard trees]\label{thm:standard-tree}
    Every full $n$-branching $\alpha$-tree is isomorphic to our $(n^{<\alpha},\subseteq)$.
\end{theorem}
\begin{proof}
    Label the $n$ immediate successors of each non‑maximal node by $0,\dots,n-1$. For a node $x$ of height $\beta<\alpha$, let $f(x)\in n^{\beta}$ be the sequence of labels along the path from the root to $x$. Transfinite induction on $\beta$ shows that $f$ maps level $\beta$ bijectively onto $n^{\beta}$, and $x<y$ holds exactly when $f(x)$ is a proper initial segment of $f(y)$. Notice that uniqueness of $f(x)$ when $x$ is a limit height comes from the fact that every two nodes admits a greatest common lower bound. Hence, $f$ is the required isomorphism.
\end{proof}

Finitely colored ordinals play an important role in the analysis of FnBTs. We abbreviate it as FCOs. Given $\mathfrak T_\alpha^n$ and a node $s\colon\beta\to n$, we may regard $s$ as the $n$-colored ordinal $(\beta,C_0,\dots,C_{n-1})$, where $\gamma\in C_i$ if and only if $s(\gamma)=i$.

At last, let us take a glance at the sledgehammer-like theorems in the monadic second order theory of ordinals. Monadic Second order theory, abbreviated as MSO, allows second order quantification over unary predicates. There are two main approaches to these results: one rooted in model theory, mainly pushed forward by Shelah; the other grounded in automata theory, which originated from B\"uchi and Rabin. We will focus on the model-theoretic approach here.\label{subsec:hardcoreprelim}

A classical decidability theorem due to B\"uchi and Siefkes \cite{Buchi1973} states:

\begin{theorem}\label{theorem:Decidability}
    Suppose that $\alpha$ is an ordinal, $\alpha<\omega_2$; then MSO theory over $(\alpha,<)$ is decidable.
\end{theorem}

The decidability of the MSO theories of all ordinals is known to be independent of $\mathsf{ZFC}$. It remains open whether this question can be settled in $\mathsf{ZFC}+\mathsf{V=L}$ together with the assertion ``there is no weakly compact cardinal.''; set-theoretic techniques established former indepence, but They do not address the latter.  See \cite{Shelah1975} for background. The following theorem plays a role in Shelah's model-theoretic reconstruction.

For an ordering $I$, a coloring $f\colon[I]^2\to C$, where $C$ is finite, is called additive if $f(x,y)=f(x',y')$ and $f(y,z)=f(y',z')$ imply $f(x,z)=f(x',z')$. We use the following Ramsey theorem for additive colorings.
\begin{theorem}\label{theorem: Additivity Ramsey}
    If $\kappa$ is a limit ordinal, and $f$ is an additive coloring of $\kappa$, then there is an unbounded homogeneous subset $J\subseteq\kappa$ of $f$.
\end{theorem}
This theorem is important because, generally, we can't get a homogeneous set from $\kappa\to[\kappa]^2_2$, which is actually a definition of weakly compact when $\kappa$ is uncountable. We note that the $q$-characters form an additive coloring by Lemma~\ref{lem:colored-facts}. In Shelah's work\cite{Shelah1975}, this combinatorial theorem plays a central role in translate theorems and proofs in automata theory into a model-theoretic version. We hope it could also provide a model-theoretic proof of Theorem~\ref{lem:exact-stretching}, but we will use a more simple proof in the context of ordinal automaton.

The finite ordinal-profile scheme below is a consequence of Kamp's theorem
\cite{Kamp1968}.  We also use the simple ordinal-automaton construction of
Demri and Rabinovich \cite[Theorem~1.1, Definition~2.1, and
Lemma~2.3]{DemriRabinovich2010}. Ordinal automata are particularly useful for constructing structures of specified type and length because they generate infinite sequences from finite transition data.

\begin{theorem}[Finite ordinal-profile scheme]
\label{lem:finite-profile-scheme}
Let
$
  L_m=\{<,P_0,\ldots,P_{m-1}\}$
be a finite relational language with a linear order and unary predicates,
and let $\varphi$ be an $L_m$-sentence.  There are finite data
\[
 \mathcal A_\varphi=(D,Q,I,F_{\mathrm{succ}},F_{\mathrm{lim}},
                      \delta_{\mathrm{next}},\delta_{\mathrm{lim}}),
  Q\subseteq\mathcal P(D),
\]
where $D$ is finite,
\[
 I,F_{\mathrm{succ}}\subseteq Q,\qquad
 F_{\mathrm{lim}}\subseteq\mathcal P(D),\qquad
 \delta_{\mathrm{next}}\subseteq Q^2,\qquad
 \delta_{\mathrm{lim}}\subseteq\mathcal P(D)\times Q,
\]
and a decoding map $\ell:Q\to\{0,1\}^{m}$ with the following property.
For a map $r:\alpha\to Q$ and a nonzero limit ordinal
$\delta\leq\alpha$, define
\[
 \operatorname{LimSupp}(r,\delta)=
 \{e\in D:\exists\beta<\delta\ \forall\xi\,
       (\beta<\xi<\delta\longrightarrow e\in r(\xi))\}.
\]
For every nonzero ordinal $\alpha$ and every $L_m$-structure
\[
 \mathfrak M=(\alpha,<,P_0^{\mathfrak M},\ldots,
                         P_{m-1}^{\mathfrak M}),
\]
we have $\mathfrak M\models\varphi$ if and only if there is a map
$r:\alpha\to Q$ such that
\begin{enumerate}
  \item $r(0)\in I$, and
  $\ell(r(\xi))_i=1$ if and only if
  $\xi\in P_i^{\mathfrak M}$, for every $\xi<\alpha$ and $i<m$;
  \item $(r(\xi),r(\xi+1))\in\delta_{\mathrm{next}}$ whenever
  $\xi+1<\alpha$;
  \item $(\operatorname{LimSupp}(r,\delta),r(\delta))\in
  \delta_{\mathrm{lim}}$ for every nonzero limit
  $\delta<\alpha$;
  \item if $\alpha=\beta+1$, then $r(\beta)\in F_{\mathrm{succ}}$,
  and if $\alpha$ is a limit ordinal, then
  $\operatorname{LimSupp}(r,\alpha)\in F_{\mathrm{lim}}$.
\end{enumerate}
The data are finite and effectively obtainable from $\varphi$.
\end{theorem}

\section{A sharper EF-bound for FCOs}\label{sec:EFbound}

In this section, we establish a refined upper bound for the EF bound of finitely colored ordinals and a corresponding lower bound. The two bounds lie in the same tower-exponential growth class. We then give an algorithm that computes the exact value. This work builds on \cite{Mwesigye2011,Mwesigye2018}.

For each $d\geq 1$ and $q<\omega$, fix a set $\Sigma_d(q)$ containing exactly one $q$-characteristic sentence (in the general sense) for each $q$-equivalence class of $d$-colored ordinals. Empty domains and empty color classes are allowed. For any of these sets, a superscript $+$ denotes the subset consisting of those characteristic sentences whose $q$-equivalence class has a nonempty representative. This set is finite. The cardinality of $\Sigma_d(q)$ is the number of $q$-equivalence classes of $d$-colored ordinals. We write $\Sigma_d^{\mathrm{fin}}(q)$ for the classes having finite representatives and $\Sigma_d^{\mathrm{ord}}(q)$ for the classes having infinite representatives.
Mwesigye only considered $\Sigma_d^{\mathrm{fin}}(q)$ and defined his $g$ function in \cite{Mwesigye2011}, so here we put
$$g_{\mathrm{ord}}(d,q)=\max_{\sigma\in\Sigma_d(q)}\min\{\operatorname{otp}(A):\mathfrak A\models\sigma\}.$$

We also recall the quantitative conclusion of the first-occurrence argument in \cite[Theorem~3.4]{Mwesigye2018}. If $q\geq2$, $\mathfrak A$ has least order type in its $q$-equivalence class, and $k=|\rho_{q-1}(\mathfrak A)|$, then
$$\operatorname{otp}(A)<\omega^k\cdot2k^2.$$
Our goal is to improve this estimate.

\subsection{Refining the EF-bound for finitely colored ordinals}

An $q$-character is \emph{cofinal at a limit cut} $\lambda$ when its occurrences below $\lambda$ are cofinal in $\lambda$.

We recall Lemmas~\ref{lem:cutting} and~\ref{lem:limit-cutting}
\cite[Lemma~3.3]{Mwesigye2018}.

\begin{lemma}[cutting lemma]\label{lem:cutting}
Let $\mathfrak A$ be a $d$-colored ordinal, and let $(a_1,a_2)$ and $(b_1,b_2)$ be intervals in $A$. Suppose that
\begin{itemize}
    \item $a_1$ and $b_1$, and respectively $a_2$ and $b_2$, have the same $q$-character;
    \item the sets of $q$-characters realized in $(a_1,a_2)$ and $(b_1,b_2)$ are both equal to $C$;
    \item each interval can be partitioned into $2^q-1$ disjoint intervals, every one of which realizes every member of $C$.
\end{itemize}
Then $(a_1,a_2)\equiv_q(b_1,b_2)$.
\end{lemma}

\begin{lemma}[Limit cutting]\label{lem:limit-cutting}
    Fix $q\geq2$, put $r = q-1$, and let $\mathfrak A$ be a $d$-colored ordinal. Let $[a_\xi,b_\xi)(\xi<\Lambda)$ be intervals with limit positions as endpoints. Assume that they are strictly separated, meaning that
    $$ \xi<\eta\Longrightarrow a_\xi<b_\xi<a_\eta,\qquad\eta\text{ limit}\Longrightarrow\sup_{\xi<\eta}b_\xi<a_\eta.$$
    
    For each $\xi<\Lambda$, suppose that the $r$-characters occurring cofinally below $a_\xi$ and $b_\xi$ form the same set $X_\xi$, and for every $a\in[a_\xi,b_\xi)$, we have $\rho_r^{\mathfrak A}(a)\in X_\xi$.
    
    Let $B=A\setminus\bigcup_{\xi<\Lambda}[a_\xi,b_\xi)$. Then $\mathfrak A\mathbin{\upharpoonright}B\equiv_q\mathfrak A$.
\end{lemma}

We also recall the interval partition from \cite[Theorem~3.4]{Mwesigye2018}. Enumerate the first occurrences of the realized $r$-characters as $x_0<\cdots<x_{k-1}$, where $k$ is the number of realized $r$-characters, and introduce the formal endpoint $x_k=+\infty$. Put $I_i=[x_i,x_{i+1})$. Each $I_i$ is further partitioned into intervals $J_{i,j}$ using the ordered suprema of the occurrence sets of its $r$-characters. Lemma~\ref{lem:cutting} reduces each $J_{i,j}$ to either a singleton or an interval with a limit right endpoint, while Lemma~\ref{lem:limit-cutting} removes redundant overlaps in a representative of least order type.

\begin{lemma}\label{lem:supported-block}
    Fix $q\geq2$, put $r=q-1$, and let $\mathfrak A$ be a $d$-colored ordinal of least order type in its $q$-equivalence class. Let $J$ be one of the convex intervals $J_j$ obtained in the proof of \cite[Theorem~3.4]{Mwesigye2018}, and put
    $$X=\{\rho_r^{\mathfrak A}(a):a\in J\},\qquad s=|X|\geq1.$$
    
    Suppose that every member of $X$ occurs cofinally at the right endpoint of $J$, and that when $s=1$, $\operatorname{otp}(J)$ is either $1$ or a limit ordinal. There is $B\subseteq A$ such that $\mathfrak B=\mathfrak A\mathbin{\upharpoonright}B\equiv_q\mathfrak A$, $\operatorname{otp}(B\cap J)\leq\omega^s$, and the choices may be made simultaneously for all the finitely many intervals $J_j$.
\end{lemma}
\begin{proof}
    Keep all the preceding steps of the induction at the end of the proof of \cite[Theorem~3.4]{Mwesigye2018}. Before its last step for $X$, it has already shown that, for every nonempty proper subset $Y$ of $X$, every convex $H\subseteq J$ exhibiting only $r$-characters in $Y$ satisfies $\operatorname{otp}(H)<\omega^{|Y|}\cdot2$.
    
    Suppose that $\operatorname{otp}(J)>\omega^s$, and let $Y$ be the set of $r$-characters occurring cofinally below the relative position $\omega^s$ in $J$. The set $Y$ is nonempty. If $Y$ were a proper subset of $X$, the occurrences below $\omega^s$ of the finitely many members of $X\setminus Y$ would have a common bound. Above this bound, there would be a convex interval $H$ exhibiting only members of $Y$ and having order type $\omega^s$, contrary to
    $$\operatorname{otp}(H)<\omega^{|Y|}\cdot2\leq\omega^{s-1}\cdot2<\omega^s.$$

    Hence $Y=X$. Every member of $X$ also occurs cofinally at the right endpoint of $J$. If this endpoint is internal to $A$, Lemma~\ref{lem:limit-cutting} removes the interval between these two positions. For a terminal $J$, use the terminal form of the same argument, obtained by adjoining one last point of a fresh color and removing it after the cut. In either case, the retained part of $J$ has order type $\omega^s$. The simultaneous form gives the required $B$ for all the intervals $J_j$.
\end{proof}

\begin{theorem}[Refined bound]\label{thm:refined-bound}
    Let $q\geq2$, $k=|\rho_{q-1}(\mathfrak A)|$. For every $\mathfrak A$, there is a $\mathfrak B$ that is $q$-equivalent to $\mathfrak A$ and $\operatorname{otp}(B)\leq\omega^k$.
\end{theorem}
\begin{proof}
    The empty case is immediate. Replace $\mathfrak A$ by a representative of least order type in the $q$-equivalence class of $\mathfrak A$. The value of $k$ is unchanged by Lemma~\ref{lem:colored-facts}. Put $r=q-1$ and retain the notation $I_i$ and $J_{i,j}$ from the proof of \cite[Theorem~3.4]{Mwesigye2018}. At most $i+1$ different $r$-characters occur in $I_i$, and the nonempty sets of $r$-characters occurring in its intervals $J_{i,j}$ are pairwise disjoint. By Lemma~\ref{lem:supported-block}, an interval in which $u$ different $r$-characters occur has order type at most $\omega^u$.
    
    If one interval in $I_i$ realizes all $i+1$ characters, it is the only one. Otherwise, every interval realizes at most $i$ characters, and their finite sum is below $\omega^{i+1}$. The resulting $\mathfrak B\equiv_q\mathfrak A$ therefore satisfies
    $$\operatorname{otp}(B)\leq\sum_{i<k}\omega^{i+1}=\omega^k.$$
    
    The minimality of $\mathfrak A$ gives the same bound for $\operatorname{otp}(A)$.
\end{proof}

We also note that some small q-characters only appear once in the initial segments.
\begin{lemma}[Small orders only appear once]\label{lem:small-chain}
    If $r\geq2$ and $a<2^r-1$, then the finite order of length $a$ is not $r$-equivalent to any longer order.
\end{lemma}
\begin{proof}
    We use induction on $r$. The assertion for $r=1$ compares the empty order with a nonempty order. For the induction step from $r$ to $r+1$, put $t=2^r-1$ and compare orders of lengths $a<b$ with $a<2^{r+1}-1=2t+1$. If $a<t$, the induction hypothesis already distinguishes the orders. Otherwise, $\forall$ chooses a point with exactly $t$ predecessors. Every response in the shorter order falls into the induction hypothesis.
\end{proof}

\begin{corollary}[Improved upper bound]\label{thm:improved-bound}
Let $q\geq2$, and let $\mathfrak A$ have the least order type in its
$q$-equivalence class. Put $k=\lvert\rho_{q-1}(\mathfrak A)\rvert$. Then
$$k\leq 2^{q-1}-1\Longrightarrow\operatorname{otp}(A)=k,\qquad k>2^{q-1}-1\Longrightarrow\operatorname{otp}(A)\leq\omega^{k-2^{q-1}+1}.$$
\end{corollary}

\subsection{Lower bounds}

In this section, we first construct, for $d\geq2$, a $d$-colored ordinal whose equivalence class has no representative of order type below $\omega^{|\Sigma_{d-1}(q)|}$. The binary case requires a separate encoding. We note that, our purpose is to determine the growth class of the lower bound rather than its optimal value.

\begin{lemma}[Lower bound construction]\label{lem:marker-hierarchy}
    For $d\geq2$ and $q\geq1$, there is a $d$-colored ordinal
    $\mathfrak V$ such that every $d$-colored ordinal
    $\mathfrak B\equiv_{q+4}\mathfrak V$ satisfies
    $$\operatorname{otp}(B)\geq\omega^{|\Sigma_{d-1}(q)|}.$$
  \end{lemma}
\begin{proof}
    Reserve $C_{d-1}$ as a delimiter color and enumerate $\Sigma_{d-1}(q)=\{\sigma_1,\ldots,\sigma_s\}$. Choose a representative $\mathfrak U_i\models\sigma_i$ for each $i$. A delimiter point $x$ is marked by $P_i$ when the open interval from $x$ to the next delimiter satisfies $\sigma_i$. The formula $P_i(x)$ has rank at most $q+1$, and $P_i$ are pairwise disjoint.
    
    Let $\mathfrak M_i$ be a delimiter point followed by $\mathfrak U_i$, and define
    $$\mathfrak V_1=\mathfrak M_1^\omega,\qquad \mathfrak V_{i+1}=(\mathfrak V_i+\mathfrak M_{i+1})^\omega.$$
    The $P_i$-points are cofinal in $V_i$, and every $P_i$-point for $i>1$ is a nonzero limit of $P_{i-1}$-points. More explicitly, put
    \[\begin{split}
        P_i(x):={}&C_{d-1}(x)\wedge\exists y\bigl(x<y\wedge C_{d-1}(y)\\
        &\wedge\forall z\,(x<z<y\longrightarrow\neg C_{d-1}(z))\wedge (\sigma_i)^{(x,y)}\bigr),
    \end{split}\]
    where the superscript denotes relativization to $(x,y)$. The required properties are expressed by
    \[\begin{split}
        \Phi_s:={}&\exists w\,(w=w)\wedge\forall x\,\exists y\,(x<y\wedge P_s(y))\wedge\bigwedge_{i=2}^s\forall x\Bigl(P_i(x)\longrightarrow\\
        &\Bigl[\exists u\,(u<x)\wedge\forall y\bigl(y<x\longrightarrow\exists z\,(y<z<x\wedge P_{i-1}(z))\bigr)\Bigr]\Bigr).
    \end{split}\]
    Its quantifier rank is at most $q+4$.

    For an ordinal $\beta$, define $\Delta_0(\beta)=\beta$ and let $\Delta_{j+1}(\beta)$ be the nonzero limit points of $\Delta_j(\beta)$. For $j\geq1$, induction gives
    $$\Delta_j(\beta)=\{\omega^j\gamma<\beta:\gamma>0\}.$$
    Let $s=|\Sigma_{d-1}(q)|\geq2$. The sentence $\Phi_s$ implies that $\Delta_{s-1}(\beta)$ is cofinal in $\beta$, whence $\beta\geq\omega^s$. This applies to the domain of every colored ordinal $(q+4)$-equivalent to $\mathfrak V_s$. Take $\mathfrak V=\mathfrak V_s$.
\end{proof}

We next encode both the delimiter and the intervening data using two colors. Write the colors as $0$ and $1$,
$\mathtt d=10,\qquad\operatorname{enc}(0)=110,\qquad\operatorname{enc}(1)=1110,$
by defining the following formula
\[\begin{split}
  \operatorname{Succ}(x,y):= {}&x<y\wedge\neg\exists z\,(x<z<y),\\
  L(x):={}&\exists u\bigl(C_1(u)\wedge\operatorname{Succ}(u,x)\bigr),\\  R(x):={}&\exists v\bigl(C_1(v)\wedge\operatorname{Succ}(x,v)\bigr),\\
  Q_D(x):={}&C_1(x)\wedge\neg L(x)\wedge\neg R(x),\\
  Q_0(x):={}&C_1(x)\wedge\neg R(x)\wedge\exists u\bigl(C_1(u)\wedge\operatorname{Succ}(u,x)\wedge\neg L(u)\bigr),\\
  Q_1(x):={}&C_1(x)\wedge\exists u\bigl(C_1(u)\wedge\operatorname{Succ}(u,x)\wedge\neg L(u)\bigr)\\
  &{}\wedge\exists v\bigl(C_1(v)\wedge\operatorname{Succ}(x,v)\wedge\neg R(v)\bigr).
\end{split}\]

Thus the color predicates are replaced by $Q_0$ and $Q_1$, while the delimiter predicate is replaced by $Q_D$. These formulas have quantifier rank at most $3$, so the interpretation increases the quantifier rank by at most $3$.

\begin{lemma}[Binary construction]\label{lem:binary-markers}
    For every $q\geq1$, there is a binary colored ordinal $\mathfrak W$ with $\operatorname{otp}(W)\geqslant \omega^{|\Sigma_2(q)|}$ such that every binary colored ordinal $\mathfrak B\equiv_{q+7}\mathfrak W$ has order type at least $\omega^{|\Sigma_2(q)|}$.
\end{lemma}
\begin{proof}
  Same as the normal construction.
\end{proof}

In summary, we have:

\begin{theorem}[Lower bounds]\label{thm:lower-bounds}
    The following estimates hold:
    $$g_{\mathrm{ord}}(d,q)\geq\omega^{|\Sigma_{d-1} (q-4)|}\qquad(d\geq2,\ q\geq5),$$
    $$g_{\mathrm{ord}}(d,q)\geq\omega^{|\Sigma_2 (q-7)|}\qquad(d\geq2,\ q\geq8).$$
\end{theorem}

\subsection{Tower-exponential growth}

We first make the upper bound uniform. An FCO $(q-1)$-character has three components: a color, a left $(q-1)$-equivalence class, and a right $(q-1)$-equivalence class. Hence the number $k$ of realized $(q-1)$-characters in Theorem~\ref{thm:refined-bound} is at most $d|\Sigma_d(q-1)|^2$. Corollary~\ref{thm:improved-bound} therefore gives, for $q\geq2$,
\begin{equation}
    g_{\mathrm{ord}}(d,q)\leq\omega^{d|\Sigma_d(q-1)|^2-2^{q-1}+1}.
    \label{eq:uniform-upper}
\end{equation}

By Lemma~\ref{lem:colored-facts}(\ref{itm:colored-character}), each $(q+1)$-equivalence class is determined by the set of $q$-characters it realizes. There are at most $d|\Sigma_d(q)|^2$ possible $q$-characters, and therefore $|\Sigma_d(q+1)|\leq2^{d|\Sigma_d(q)|^2}$. Hence the exponent in \eqref{eq:uniform-upper} grows at most tower-exponentially with $q$.

Now we characterize the growth of the lower bound.

\begin{lemma}\label{lem:finite-word-count}
    For $d,q\geq1$, $|\Sigma_d (q)|\geq\sum_{a=0}^{2^q-1}d^a$.
\end{lemma}
\begin{proof}
    The estimate follows from the standard EF argument: colored orders of different lengths $a,b\leq2^q-1$ are distinguishable in $q$ rounds. Thus no order of length $a\leq2^q-1$ is $q$-equivalent to an order of a different length. Following the idea in Lemma~\ref{lem:small-chain} gives a win in the remaining $q-1$ rounds. There are $d^a$ colorings of length $a$, and summing over $0\leq a\leq 2^q-1$ proves the claim.
\end{proof}

\begin{lemma}[The growth]\label{lem:binary-powerset}
    For every $q\geq1$, $|\Sigma_2 (q+5)|\geq2^{|\Sigma_2 (q)|}$.
\end{lemma}
\begin{proof}
Let $s=|\Sigma_2(q)|$, and choose a representative colored ordinal $\mathfrak U_i$ from each $q$-equivalence class. For each $S=\{i_1<\cdots<i_k\}\subseteq\{1,\dots,s\}$, define
$E_S=\mathtt d\operatorname{enc}(\mathfrak U_{i_1})\mathtt d\cdots\mathtt d\operatorname{enc}(\mathfrak U_{i_k})\mathtt d$
and $E_\varnothing=\mathtt d$. The resulting structures are distinguished by sentences of rank at most $q+5$: $\mathfrak E_S\models\exists x\,\widehat P_i(x)$ if and only if $i\in S$, where $\widehat P_i(x)$ expresses that the block following $x$ encodes a representative of the class of $\mathfrak U_i$.
\end{proof}

As Lemma~\ref{lem:binary-powerset} also ensures that the lower bound a tower-exponential growth, the upper and lower bounds for $g_{\mathrm{ord}}(d,q)$ both have the height of the tower linear in $q$. This means that, although the upper bound may not be optimal, as Mwesigye and Truss note at the end of the article, the same growth level as the lower bound suggests that substantially sharpening it may be difficult.

\subsection{An exact computation algorithm}

The bounds in Theorem~\ref{thm:lower-bounds} and \eqref{eq:uniform-upper} admit an exact finite refinement. For $\sigma,\tau\in\Sigma_d(q)$, choose colored ordinals $\mathfrak A,\mathfrak B$ satisfying them. Concatenation and $\omega$-fold repetition induce
$$\sigma\oplus\tau\equiv\sigma_{\mathfrak A+\mathfrak B}^q,\qquad\Omega(\sigma)\equiv\sigma_{\mathfrak A\cdot\omega}^q,
$$
where the left sides denote the chosen logically equivalent members of $\Sigma_d(q)$. These operations are well defined by Lemma~\ref{lem:colored-facts}(\ref{itm:colored-sum}).

\begin{lemma}[Effective finite characteristic algebra]\label{lem:effective-characteristic-algebra}
    For fixed $d$ and $q$, the members of $\Sigma_d(q)$ have canonical finite codes with decidable equality, and the operations $\oplus$ and $\Omega$ are effectively computable on those codes.
\end{lemma}
\begin{proof}
    The codes represent the characteristic sentences defined in Section~\ref{sec:preliminaries}. During the recursion, write $\oplus_r$ and $\Omega_r$ for the operations on rank-$r$ codes. Put $\mathbf c_0(\mathfrak A)=*$ and
    $$\mathbf c_{r+1}(\mathfrak A)=\left(\mathbf c_r(\mathfrak A),\left\{\bigl(i,\mathbf c_r(\mathfrak A_{<a}),\mathbf c_r(\mathfrak A_{>a})\bigr):a\in A\right\}\right).$$
    
    At each rank there are only finitely many possible codes, and Lemma~\ref{lem:colored-facts}(\ref{itm:colored-character}) shows that $\mathbf c_r(\mathfrak A)=\mathbf c_r(\mathfrak B)$ exactly when $\mathfrak A\equiv_r\mathfrak B$.

    Suppose $\mathbf c_{r+1}(\mathfrak A)=(s,\mathcal C)$ and $\mathbf c_{r+1}(\mathfrak B)=(t,\mathcal D)$. Once the rank-$r$ addition table is known, the second component of the code for $\mathfrak A+\mathfrak B$ is
    $$\{(i,\ell,u\oplus_r t):(i,\ell,u)\in\mathcal C\}\cup\{(i,s\oplus_r\ell,u):(i,\ell,u)\in\mathcal D\}.$$
    
    This computes $\oplus_{r+1}$ recursively.

    For $\Omega_{r+1}$, let $\mathbf 0$ be the empty colored order, put $p_0=\mathbf c_r(\mathbf 0)$, and $p_{j+1}=p_j\oplus_r s$, and let $\operatorname{Orb}_r(s)=\{p_j:j<\omega\}$. The orbit is effectively found by iteration until its first repetition. If $w=\Omega_r(s)$, then the second component of the code for $\mathfrak A\cdot\omega$ is
    $$\{(i,p\oplus_r\ell,u\oplus_r w):(i,\ell,u)\in\mathcal C,\ p\in\operatorname{Orb}_r(s)\}.$$
    
    Together with its rank-$r$ projection $w$, this is $\mathbf c_{r+1}(\mathfrak A\cdot\omega)$. Starting at rank zero therefore computes both operations and makes equality a comparison of finite codes.
\end{proof}

We now give an algorithm which can compute the exact value of $g_{ord}$.

\begin{lemma}[Length-preserving rational form]\label{lem:rational-form}
    Every $d$-colored ordinal $\mathfrak A$ has a $q$-equivalent representative $\mathfrak R$ given by a finite term
    $$\mathsf R::=\varnothing\mid c_i\mid\mathsf R+\mathsf R\mid \mathsf R^\omega$$
    where $c_i$ is the one-point order of color $i$, and whose order type is at most $\operatorname{otp}(A)$.
\end{lemma}
\begin{proof}
If $A$ is uncountable, the downward L\"owenheim--Skolem theorem gives a countable elementary substructure $\mathfrak B\preccurlyeq\mathfrak A$. So it suffices to consider countable ordinals. We argue by induction on the order type. The empty case is represented by $\varnothing$, and the successor case is obtained by adjoining the final colored point.

Suppose that $\operatorname{otp}(A)=\lambda$ is a countable limit ordinal, and choose cuts $0=\gamma_0<\gamma_1<\cdots$ cofinal in $\lambda$. Color each pair $\{i,j\}$, with $i<j$, by the rank-$q$ characteristic of $\mathfrak A\mathbin{\upharpoonright}[\gamma_i,\gamma_j)$. Infinite Ramsey's theorem gives homogeneous indices $h_0<h_1<\cdots$. Write
\[
\mathfrak A=\mathfrak P_0+\mathfrak B_0+\mathfrak B_1+\cdots,
\]
 Lemma~\ref{lem:colored-facts}(\ref{itm:colored-sum}) gives
\[
\mathfrak A\equiv_q\mathfrak P+\mathfrak E^\omega.
\]
Moreover, $\operatorname{otp}(E)\leq\operatorname{otp}(B_j)$ for every $j$, and hence
$$
\operatorname{otp}(P)+\operatorname{otp}(E)\cdot\omega
\leq\operatorname{otp}(P_0)+\sum_{j<\omega}\operatorname{otp}(B_j)
=\operatorname{otp}(A).
$$
\end{proof}

\begin{theorem}[Exact finite algorithm]\label{thm:exact-algorithm}
    For every $d\geq1$ and $q<\omega$, the exact value $g_{\mathrm{ord}}(d,q)$ is computable by a terminating finite algorithm.
\end{theorem}
\begin{proof}
    Start with the characteristics of the empty order and the $d$ one-point orders, and close under $\oplus$ and $\Omega$, using the codes from Lemma~\ref{lem:effective-characteristic-algebra}. Store a witnessing rational term for every state. The work queue terminates because only finitely many rank-$q$ codes exist, and Lemma~\ref{lem:rational-form} shows that the resulting closure is exactly $\Sigma_d(q)$.

    Let $\ell(\sigma)$ denote the current ordinal label of a state. Initialize each state with the order type of its stored witnessing term; if several initial atoms have the same state, take their least length. In particular, at rank zero the common state receives the empty witness of length $0$; taking the minimum discards the one-point length $1$. Fairly scan all pairs and all states, performing every strict relaxation
    $$\ell(\sigma\oplus\tau)\leftarrow\min\bigl(\ell(\sigma\oplus\tau),\ell(\sigma)+\ell(\tau)\bigr),$$
    $$\ell(\Omega(\sigma))\leftarrow\min\bigl(\ell(\Omega(\sigma)),\ell(\sigma)\cdot\omega\bigr).$$
    
    Every label carries a witnessing term. The order type of a finite rational term lies below $\omega^\omega$ and is represented effectively in Cantor normal form.  If the scan did not terminate, one of the finitely many states would undergo infinitely many strict decreases, contradicting the well-foundedness of the ordinals.

    At termination, witnesses show that the least order type of a representative of $\sigma$ is at most $\ell(\sigma)$. Induction on rational terms, using the two stable inequalities, gives $\ell(\sigma)\leq\operatorname{otp}(R)$ for every term $\mathsf R$ whose value has domain $R$ and satisfies $\sigma$. Lemma~\ref{lem:rational-form} gives the reverse inequality. Hence
    $$g_{\mathrm{ord}}(d,q)=\max_{\sigma\in\Sigma_d(q)}\ell(\sigma).$$
\end{proof}

We note that computing these exact values is difficult in practice, maybe like Ramsey numbers. For example, computing $g_{\mathrm{ord}}(2,3)$ requires traversing a space of approximately $2^{80,802}$ states. The exact values in the first two ranks are $g_{\mathrm{ord}}(d,1)=d$ and $g_{\mathrm{ord}}(d,2)=\omega\cdot d$.

\section{Realizability of FCO characters}\label{sec:realizability}

In Section~\ref{sec:EFbound} we established that every $q$-characteristic of a finitely colored ordinal is realized on some ordinal whose order type does not exceed the effectively computable bound $g_{\mathrm{ord}}(d,q)<\omega^\omega$, and that this bound grows tower-exponentially with $q$.  This provides an absolute upper bound on the length of the shortest witness of a given character, but it leaves open the finer question: 

\emph{For a given ordinal $\alpha$ and a $q$-characteristic $\sigma$, does there exist a $d$-colored ordinal of order type exactly $\alpha$ that satisfies $\sigma$?}

Equivalently, we wish to determine the set of all $q$-characteristics for each ordinal $\alpha$ that are realizable on $\alpha$, and, conversely, to describe the structural properties of the ordinals that can host a given $\sigma$.  In particular, we are interested in the decision problem: given $q<\omega$, two ordinals $\theta_1<\theta_2$, and $\sigma\in\Sigma_d^+(q)$, can we effectively decide whether $\sigma$ is realized on some ordinal $\geq\theta_1$ and $<\theta_2$?  The purpose of this section is to answer all these questions positively, by developing a uniform combinatorial analysis of the realizability sets of FCO characters.

We can prove the following basic realizability theorem immediately from big theorems we have introduced.

\begin{theorem}\label{lem:interval-realizabiliy}
    Given $\theta_1<\theta_2<\omega^\omega$, $q<\omega$, and $\sigma\in\Sigma_d^+(q)$, there is an effective procedure deciding whether $\sigma$ is realizable on a $d$-colored ordinal of order type $\alpha$ for some $\theta_1\leq\alpha<\theta_2$.
\end{theorem}
\begin{proof}
    The ordinals $\theta_1$ and $\theta_2$ are first-order definable in $(\omega_1,<)$, and the stated question is expressible by an MSO sentence over this structure. The conclusion follows from the decidability of its MSO theory.
\end{proof}

We may wish a finer and recursive answer on all ordinals, instead of this effective but uninteresting proof restricted to countable ordinals.

\subsection{Universal ordinals and cofinal stretching}

As we have noticed before, for every $q<\omega$ there is an additively indecomposable ordinal $b_q<\omega^\omega$ such that every FCO is $q$-elementarily equivalent to a FCO with order type $<b_q$; in fact, we may let $b_q$ be the least additively indecomposable ordinal that is strictly greater than $g_{\mathrm{ord}}(d,q)$.  The existence of such $b_q$ follows immediately from Corollary~\ref{thm:improved-bound} together with the fact that the ordinals below $\omega^\omega$ are well‑ordered.

Fix any regular cardinal $\kappa$ (not necessarily uncountable; so $\kappa$ might be $\omega$).  The definition of a ``universal block'' must guarantee that the block is long enough to realize all $q$‑characters that could possibly occur in the tail of a large colored ordinal.  This leads to the following inductive definition.

An additively indecomposable ordinal $\theta$ is called \emph{$q$-universal $\omega$-cofinal} if $\cf(\theta)=\omega$ and $\theta\geq b_q$. An ordinal $\gamma$ is called \emph{$(q+1)$-universal $\kappa$-cofinal} if
\begin{itemize}
    \item $\cf(\gamma)=\kappa$,
    \item for every $\delta<\gamma$ and every positive integer $M$ there exists a strictly increasing cofinal sequence $\lambda_0<\lambda_1<\dots<\lambda_\xi<\dots\;\;(\xi<\kappa)$ with $\lambda_0>\delta$ such that each interval $[\lambda_\xi,\lambda_{\xi+1})$ can be written as a concatenation of $M$ consecutive $q$-universal $\omega$-cofinal ordinals.
\end{itemize}

Thus a $q$-universal $\omega$-cofinal ordinal is an additively indecomposable limit of $\omega$ whose length is at least $b_q$, and a $(q+1)$-universal $\kappa$-cofinal ordinal is one that can be decomposed into sufficiently many such blocks. Note that for $q=0$ we have $b_0=1$, so every additively indecomposable ordinal of cofinality $\omega$ is $0$-universal $\omega$-cofinal.

The following lemma shows that every ordinal that is closed under adding any ordinal $<\omega^\omega$ is already $q$-universal for all $q$.

\begin{lemma}[Universality of $\omega^\omega$-like ordinals]\label{lem:universal-large}
    Let $\gamma\equiv\omega^\omega$ and $\kappa=\cf(\gamma)$. Then $\gamma$ is $q$-universal $\kappa$-cofinal for every $q<\omega$.
\end{lemma}
\begin{proof}
    Fix $q$ and let $b=b_q$.  Since $b<\omega^\omega$, we can write $b=\omega^m$ for some finite $m$. Given $\delta<\gamma$ and $M$, we must produce the required cofinal sequence. For every $\alpha<\gamma$, we have $\alpha+\omega^{m+1}\cdot M < \gamma$. Hence we can construct by transfinite recursion a strictly increasing continuous cofinal sequence $(\lambda_\xi)_{\xi<\kappa}$ such that $\lambda_0>\delta$ and, for every $\xi$, $\lambda_\xi+\omega^{m+1}\cdot M\leq\lambda_{\xi+1}$. Now each interval $[\lambda_\xi,\lambda_{\xi+1})$ of order type at least $\omega^{m+1}\cdot M$ can be split into $M$ consecutive blocks of order type $\omega^{m+1}$.  Every block of order type $\omega^{m+1}$ is additively indecomposable, has cofinality $\omega$, and is $\ge b$. Thus it is a $q$-universal $\omega$-cofinal ordinal. The sequence $(\lambda_\xi)_{\xi<\kappa}$ therefore witnesses that $\gamma$ is $(q+1)$-universal $\kappa$-cofinal. Since $q$ was arbitrary, the lemma is proved for all $q\ge 1$; the case $q=0$ is immediate.
\end{proof}

\begin{theorem}[Cofinal stretching]\label{lem:cofinal-stretching}
    Suppose that $\gamma_1,\gamma_2$ are ordinals $\equiv\omega^\omega$, and $\cf(\gamma_1)=\cf(\gamma_2)=\kappa$. For any $q<\omega$ and any $d$-colored ordinal $\mathfrak{A}$ of order type $\gamma_1$ there exists a $d$-colored ordinal $\mathfrak{B}$ of order type $\gamma_2$ such that $\mathfrak{A}\equiv_q\mathfrak{B}$.
\end{theorem}
\begin{proof}
    Let $r = q+3$ and $M = 2^r-1$. By Lemma~\ref{lem:universal-large} the ordinals $\gamma_1,\gamma_2$ are $(r+1)$-universal $\kappa$-cofinal; in particular they are $r$-universal $\kappa$-cofinal as well (one may increase the block size if necessary).  We shall use the $r$-universal $\omega$-cofinal blocks provided by this definition, each of which has order type at least $b_r$ and therefore can accommodate all $r$-characteristics that could occur cofinally in a large ordinal.

    Consider the finite set $D$ of all $r$-characters that occur cofinally in $\mathfrak{A}$. Choose $\delta_0<\gamma_1$ so large that every $r$-character occurring above $\delta_0$ belongs to $D$ and is already cofinal.  Because $\gamma_1$ is $(r+1)$-universal $\kappa$-cofinal, there exists a cofinal sequence $(\lambda_\xi)_{\xi<\kappa}$ with $\lambda_0>\delta_0$ such that each $I_\xi = [\lambda_\xi,\lambda_{\xi+1})$ is a concatenation of $M$ consecutive $r$-universal $\omega$-cofinal blocks.

    Inside each of these $M$ blocks, choose a finite set of representatives for every character in $D$ (this is possible because the block length is at least $b_r$). Together with the endpoints $\lambda_\xi,\lambda_{\xi+1}$ and the $r$-types of the intervals cut out by these representatives, these choices determine a finite ``template''. There are only finitely many possible templates, so by the regularity of $\kappa$ we can thin out the sequence so that \emph{every} $I_\xi$ carries the same template. Write
    $$\mathfrak{A} = P + \sum_{\xi<\kappa} H_\xi,\qquad P = \mathfrak{A}\restriction\lambda_0,\;H_\xi = \mathfrak{A}\restriction I_\xi.$$

    Choose a cofinal sequence $(\mu_\xi)_{\xi<\kappa}$ in $\gamma_2$ such that each $J_\xi = [\mu_\xi,\mu_{\xi+1})$ is again a concatenation of $M$ consecutive $r$-universal $\omega$-cofinal blocks, and $\mu_0$ is taken large enough so that $\mu_0 >\operatorname{otp}(P)$.  Replace $P$ by an $r$-equivalent colored ordinal $P^*$ of order type $<\omega^\omega$ (possible by the boundedness theorem) and embed it into the initial segment $[0,\mu_0)$ of $\gamma_2$; the tail of $\gamma_2$ beyond $P^*$ still has order type $\gamma_2$.

    Now we replicate the template in each $J_\xi$. Inside every $r$-universal $\omega$-cofinal block of $J_\xi$, we place colored points according to the fixed template of the $H_\xi$, using the fact that the block is long enough to realize all characters in $D$. The endpoints of the blocks receive the same $r$-characters as their counterparts. This produces a colored ordinal $H'_\xi$ that is $r$-equivalent to $H_\xi$. Set
    $$\mathfrak{B} = P^* + \sum_{\xi<\kappa} H'_\xi.$$
    Clearly $\mathfrak{B}$ has order type $\gamma_2$.

    We describe a winning strategy for $\exists$ in the $r$-round EF game between $\mathfrak{A}$ and $\mathfrak{B}$. The invariant maintained after $k$ rounds ($0\le k\le r$) is:
    \begin{quote}
        The two finite tuples can be embedded into isomorphic finite linear orders of ``chunks'' with the following properties:
        \begin{itemize}
            \item the chunks on the two sides are either the initial parts $P,P^*$, or intervals inside corresponding $H_\xi$, $H'_\xi$ that are $m$-equivalent for the remaining rank $m = r-k$;
            \item inside any chunk that has not been fully explored, there remain at least $2^m-1$ completely untouched $r$-universal $\omega$-cofinal blocks on each side of the played points.
        \end{itemize}
    \end{quote}
    Initially the invariant holds with $k=0$ and $M=2^r-1$ blocks in every $H_\xi$. Suppose $\forall$ plays a new element $x$ in $\mathfrak{A}$ (the case of a move in $\mathfrak{B}$ is symmetric). If $x$ falls into $P$, $\exists$ replies using the fixed $r$-equivalence between $P$ and $P^*$. Otherwise $x$ lies in some $H_\xi$.  Let $m = r-k$ be the number of rounds remaining.  By the invariant, the active chunk inside $H_\xi$ contains at least $2^m-1$ untouched complete blocks.  The $r$-character of $x$ (with respect to the already fixed parameters) belongs to $D$.  Since every complete block realizes all characters in $D$, $\exists$ can pick a point $y$ with the same character inside the \emph{middle} complete block of the corresponding $H'_\xi$.  The middle block splits the $2^m-1$ blocks into two groups of $2^{m-1}-1$ blocks on each side, reestablishing the invariant for $m-1$ rounds.  The templates guarantee that the boundary intervals obtain matching $r$-characters.

    After $r$ rounds the chosen points satisfy the same atomic relations, so $\exists$ wins. Hence $\mathfrak{A}\equiv_r\mathfrak{B}$, and in particular $\mathfrak{A}\equiv_q\mathfrak{B}$.
\end{proof}

\subsection{Exact stretching on regular uncountable cardinals}
A striking fact is that the set of $q$-characteristics that are satisfied on every regular uncountable cardinal is actually the same for all such cardinals. To establish the result, we may assume $\kappa,\lambda$ are two of these cardinals.
\begin{theorem}[Exact stretching]\label{lem:exact-stretching}
    Let $\kappa$ and $\lambda$ be regular uncountable cardinals. For every fixed $q<\omega$ and every FCO structure $\mathfrak{A}=(\lambda,<,C_0,\dots,C_{d-1})$, there is an FCO structure $\mathfrak{B}=(\kappa,<,D_0,\dots,D_{d-1})$ such that $\mathfrak{A}\equiv_q\mathfrak{B}$.
\end{theorem}
\begin{proof}

Apply Theorem~\ref{lem:finite-profile-scheme} to $[[\mathfrak A]]^q$, and use its notation
$D,Q,I,\delta_{\mathrm{next}},\delta_{\mathrm{lim}},F_{\mathrm{lim}}$, and
$\ell$.  Since $\mathfrak A\models\varphi$, where $\varphi=[[(\mathfrak A)]]^q$, there is an accepting profile assignment $r\colon\lambda\to Q$.
Let
$
Y=\operatorname{LimSupp}(r,\lambda).
$
Then $Y\in F_{\mathrm{lim}}$.

 For $e\notin Y$, put $S_e=\{\xi<\lambda:e\notin r(\xi)\}.$
Each $S_e$ is unbounded in $\lambda$.  Since $\lambda$ is regular
uncountable, the set
$
  \operatorname{Acc}(S_e)=
  \{\delta<\lambda:S_e\cap\delta\text{ is unbounded in }\delta\}$
  is club in $\lambda$.  Hence the finite intersection
   $\cap_{e\notin Y}\operatorname{Acc}(S_e)$
Choose $\delta_0<\delta_1$ in this club above a stage after which every $e\in Y$ belongs to every $r(\xi)$. Then
$
  \operatorname{LimSupp}(r,\delta_0)=
  \operatorname{LimSupp}(r,\delta_1)=Y.
 $

The profile
condition at $\delta_0$ gives$  (Y,r(\delta_0))\in\delta_{\mathrm{lim}}.$

Let $U$ be the colored ordinal
$r\mathbin{\upharpoonright}[0,\delta_0)$, and let $V$ be the reindexed
colored interval
$r\mathbin{\upharpoonright}[\delta_0,\delta_1)$.  Both
have no last point.

We record the finite profiles of $U$ and $V$ in the finite profile language. By the bound Theorem ~\ref{thm:refined-bound}, we obtain countable copies $U^*$ and $V^*$ satisfying the corresponding profile sentences.

Now, we can consider
\[
  \mathfrak B= \mathfrak U^*+\sum_{\xi<\kappa}\mathfrak V^*_{\xi}.
\]
The profile obtained by concatenating $U^*$ with $\kappa$ copies of $V^*$ is accepted by the same automaton and decodes to $\mathfrak B$. Hence $\mathfrak B\models[[(\mathfrak A)]]^q$, and therefore $\mathfrak A\equiv_q\mathfrak B$.

\end{proof}

We tried a lot of work to avoid using automata here but eventually all failed. We think that the main obstacle is that the current character of colored ordinal is not enough to get all the data we need to construct such a colored ordinal. We need a closed unbounded set of to grasp all the things, only an unbounded set would fail to cover some intervals. A detailed analysis in the use of the automata may help avoid this, however we didn't success.

\subsection{General answer of the realizability problem}

Fix a number of colors $d$ and a quantifier rank $q$. Let $\sigma\in\Sigma_d(q)$ be a characteristic. We write $\mathfrak{A}\models\sigma$ if $\mathfrak{A}$ is a $d$-coloured ordinal whose $q$-characteristic equals $\sigma$. For an ordinal $\alpha$ let $\Psi_\sigma(\alpha)$ be the statement
\begin{equation*}\label{eq:Psi}
    \text{there exists a $d$-colored ordinal } \mathfrak{A}=(\alpha,<,C_0,\dots,C_{d-1})\text{ with } \mathfrak{A}\models\sigma .
\end{equation*}
Because the finite set of colors and the quantifier rank are fixed, $\Psi_\sigma(\alpha)$ can be expressed by a single MSO sentence in the structure $(\alpha,<)$. Indeed, $\sigma$ is the equivalence class of a first‑order sentence $\varphi_\sigma$ of rank $q$, and $\Psi_\sigma(\alpha)$ is exactly
$$(\alpha,<)\models \exists C_0\dots\exists C_{d-1}\Bigl(\bigwedge_{i<j}\forall x\,\neg(C_i(x)\wedge C_j(x))\wedge\forall x\,\bigvee_i C_i(x)\wedge \varphi_\sigma\Bigr),$$
which is a MSO sentence. In particular, if $\alpha<\omega_2$, then the truth of $\Psi_\sigma(\alpha)$ is decidable by B\"uchi's theorem (Theorem~\ref{theorem:Decidability}).

Our aim is to reduce the decision of $\Psi_\sigma(\alpha)$ for arbitrary $\alpha$ to the case where $\alpha$ is a very simple ordinal built from $\omega_1$, $\omega^\omega$, and a tail $<\omega^\omega$.  The reduction rests on the two stretching theorems proved above: the cofinal stretching theorem (Theorem~\ref{lem:cofinal-stretching}) and the exact stretching theorem (Theorem~\ref{lem:exact-stretching}).

We first note that the realizability problem respects ordinal addition via the characteristic algebra developed in Section~\ref{sec:EFbound}.

\begin{lemma}[Sum decomposition]\label{lem:sum-decomposition}
    Let $\alpha = \alpha_1 + \alpha_2$ and $\sigma\in\Sigma_d(q)$. Then $\Psi_\sigma(\alpha)$ holds if and only if there exist $\sigma_1,\sigma_2\in\Sigma_d(q)$ such that $\sigma = \sigma_1 \oplus \sigma_2$, $\Psi_{\sigma_1}(\alpha_1)$ and $\Psi_{\sigma_2}(\alpha_2)$.
\end{lemma}
\begin{proof}
    If $\mathfrak A_i\models\sigma_i$ has order type $\alpha_i$ ($i=1,2$), then $\mathfrak A_1+\mathfrak A_2$ has order type $\alpha$ and characteristic $\sigma_1\oplus\sigma_2$ by Lemma~\ref{lem:colored-facts}(\ref{itm:colored-sum}). Conversely, any colored ordinal of type $\alpha$ splits uniquely as the sum of its initial segment of type $\alpha_1$ and its final segment of type $\alpha_2$; letting $\sigma_i$ be the characteristic of the $i$-th piece gives the required decomposition.
\end{proof}

Consequently, if we can determine for each ``indecomposable large'' part of $\alpha$ the set of realizable characteristics, the whole set is obtained via the addition table of $\oplus$. The next lemma disposes of the large parts in one step.

\begin{lemma}[Absorption of large segments]\label{lem:absorb-large}
    Let $\gamma$ be an ordinal such that $\gamma\equiv\omega^\omega$ (i.e.\ its Cantor normal form contains no term below $\omega^\omega$). Let $\sigma\in\Sigma_d^+(q)$. Then the following hold.
    \begin{enumerate}
        \item If $\cf(\gamma)=\omega$, then $\Psi_\sigma(\gamma)$ holds if and only if $\Psi_\sigma(\omega^\omega)$ holds.
        \item If $\cf(\gamma)>\omega$, then $\Psi_\sigma(\gamma)$ holds if and only if $\Psi_\sigma(\omega_1)$ holds.
    \end{enumerate}
\end{lemma}
\begin{proof}
    For part (1), since $\gamma\equiv\omega^\omega$ and $\cf(\gamma)=\omega$, the ordinals $\gamma$ and $\omega^\omega$ satisfy the hypotheses of Theorem~\ref{lem:cofinal-stretching} (they are both $\equiv\omega^\omega$ and have the same cofinality $\omega$).  Taking $\gamma_1=\gamma$ and $\gamma_2=\omega^\omega$, the theorem provides for any $d$-colored ordinal on $\gamma$ an $q$-equivalent one on $\omega^\omega$, and vice versa.  Thus $\Psi_\sigma(\gamma)\iff\Psi_\sigma(\omega^\omega)$.

    For part (2), let $\kappa=\cf(\gamma)$.  Since $\gamma\equiv\omega^\omega$, its cofinality $\kappa$ is a regular uncountable cardinal and again satisfies $\kappa\equiv\omega^\omega$.  Moreover $\cf(\kappa)=\kappa$.  By Theorem~\ref{lem:cofinal-stretching} applied to $\gamma_1=\gamma$, $\gamma_2=\kappa$ we obtain $\Psi_\sigma(\gamma)\iff\Psi_\sigma(\kappa)$.  Now Theorem~\ref{lem:exact-stretching} gives exact stretching between the two regular uncountable cardinals $\kappa$ and $\omega_1$, hence $\Psi_\sigma(\kappa)\iff\Psi_\sigma(\omega_1)$.  Combining these equivalences yields $\Psi_\sigma(\gamma)\iff\Psi_\sigma(\omega_1)$.
\end{proof}

Thus, for every large block of an ordinal $\alpha$, its possible $q$-characteristics are exactly those that can be realized on the single ordinal $\omega^\omega$ (if its cofinality is $\omega$) or on $\omega_1$ (if its cofinality is uncountable).  With this reduction and the finite algebra of characteristics, the decision problem collapses to checking a finite number of MSO sentences on $\omega^\omega$ and $\omega_1$.

\begin{theorem}[Realizability is decidable]\label{cor:realizability}
    Given $\sigma\in\Sigma_d^+(q)$ and an ordinal $\alpha$ presented by its Cantor normal form, it is decidable whether there exists a $d$-coloured ordinal of order type exactly $\alpha$ satisfying $\sigma$.
\end{theorem}
\begin{proof}
    Let an ordinal $\alpha$ be given by its Cantor normal form. Write $\alpha=\alpha_0+\theta$, where $\alpha_0$ is the sum of all terms of the Cantor normal form of $\alpha$ whose exponents are at least $\omega$, and $\theta$ is the sum of the remaining terms. If $\alpha$ contains no term with exponent $\ge\omega$, set $\alpha_0=0$ and $\theta=\alpha$. In either case we have $\theta<\omega^\omega$, and if $\alpha_0>0$ then $\alpha_0\equiv\omega^\omega$. The problem is to decide whether $\Psi_\sigma(\alpha)$ holds.

    By Lemma~\ref{lem:sum-decomposition}, $\Psi_\sigma(\alpha)$ holds if and only if there exist $\sigma_1,\sigma_2\in\Sigma_d(q)$ such that $\sigma=\sigma_1\oplus\sigma_2$, $\Psi_{\sigma_1}(\alpha_0)$ and $\Psi_{\sigma_2}(\theta)$.

    If $\alpha_0=0$, then $\alpha=\theta$ and the condition reduces to $\Psi_\sigma(\theta)$. Since $\theta<\omega^\omega$, this is an MSO sentence on a countable ordinal, decidable by Theorem~\ref{theorem:Decidability}.

    If $\alpha_0>0$, then $\alpha_0\equiv\omega^\omega$. Let $\kappa=\cf(\alpha_0)$; note that $\kappa$ can be read off directly from the Cantor normal form of $\alpha_0$ (if $\beta$ is the final exponent in the Cantor normal form of $\alpha_0$, then $\operatorname{cf}(\alpha_0)=\omega$ when $\beta$ is a successor ordinal, while $\operatorname{cf}(\alpha_0)=\operatorname{cf}(\beta)$ when $\beta$ is a limit ordinal). By Lemma~\ref{lem:absorb-large},
    \begin{itemize}
        \item if $\kappa=\omega$, then $\Psi_{\sigma_1}(\alpha_0)$ holds iff $\Psi_{\sigma_1}(\omega^\omega)$ holds;
        \item if $\kappa>\omega$, then $\Psi_{\sigma_1}(\alpha_0)$ holds iff $\Psi_{\sigma_1}(\omega_1)$ holds.
    \end{itemize}
    Both $\omega^\omega$ and $\omega_1$ are strictly smaller than $\omega_2$, and therefore, by Theorem~\ref{theorem:Decidability}, the truth of the MSO sentences $\Psi_{\sigma_1}(\omega^\omega)$ and $\Psi_{\sigma_1}(\omega_1)$ is decidable. The truth of $\Psi_{\sigma_2}(\theta)$ is likewise decidable because $\theta<\omega^\omega\le\omega_1<\omega_2$.

    Since the set $\Sigma_d(q)$ is finite, there are only finitely many pairs $(\sigma_1,\sigma_2)$ with $\sigma_1\oplus\sigma_2=\sigma$. Hence we can effectively enumerate all such decompositions and check whether for at least one pair both $\Psi_{\sigma_1}(\alpha_0)$ and $\Psi_{\sigma_2}(\theta)$ hold. This yields an effective decision procedure for $\Psi_\sigma(\alpha)$.
\end{proof}

The distinction between countable and uncountable cofinality is essential. Consider the $2$-colored ordinal $(\omega^\omega,<,C_0,C_1)$, where $C_0=\{\omega^n:n<\omega\}$ and $C_1=\omega^\omega\setminus C_0$. The set $C_0$ is unbounded, $1$ is its least element, and every other $C_0$-point has an immediate predecessor within $C_0$. Consequently, every ordinal admitting a coloring with the same $4$-characteristic has countable cofinality.

\section{The classification of all FnBTs}

We now extend our analysis to the full tree structure. We prepare technical lemmas that serve as tree-theoretic analogues of the FCO $q$-characters introduced above. We then prove quantifier-elimination theorems for each canonical FnBT and establish decidability of its theory.

After analyzing the canonical FnBTs, we turn to arbitrary standard trees. Our central result states that every standard tree is elementarily equivalent to a canonical FnBT. After treating each canonical type by establishing the theory, we immediately identify the standard trees belonging to that type.

Let
$$\mathcal L_{n,S}=\{\subseteq,S_0,\ldots,S_{n-1}\},\qquad\mathcal L_n^{\mathrm{rel}}=\{\subseteq,E_0,\ldots,E_{n-1}\},\qquad\mathcal L_n^\wedge=\mathcal L_n^{\mathrm{rel}}\cup\{\wedge,\epsilon\}.$$
For the composition argument, we use the definitionally equivalent relational
presentation in which
$$E_i(x,y)\quad\Longleftrightarrow\quad y=S_i(x),$$
and write $\mathfrak T_\alpha^{n,\mathrm{rel}}$ for the resulting structure on $\mathfrak{T}_\alpha^n$.
Finite meet closures need not be closed under the total functions $S_i$; the
relational presentation keeps the finite EF positions used below. Let
$\mathfrak T_\alpha^{n,\wedge}$ be the expansion by the root $\epsilon$ and the
longest-common-initial-segment function $\wedge$. The root and meet are
definable from the tree order.

\phantomsection\label{def:path-predicates}
If $x\subsetneq y$, define
$$I(x,y)=\{z:x\subseteq z\subsetneq y\},\qquad C_i(z;y)\Longleftrightarrow\exists u\bigl(E_i(z,u)\wedge u\subseteq y\bigr).$$

Let $\mathfrak I(x,y)$ denote the FCO $(I(x,y),\subsetneq,(C_i(\mathord\cdot;y))_{i<n})$. In $\mathfrak T^{n,\mathrm{rel}}_\alpha$ it records the transition word from $x$ to $y$.

For every $q<\omega$ and every $\sigma\in\Sigma_n^{+}(q)$, introduce a binary predicate $P_{q,\sigma}$, interpreted by
$$P_{q,\sigma}(x,y)\quad\Longleftrightarrow\quad x\subsetneq y\ \text{ and }\ \mathfrak I(x,y)\models\sigma.$$

Let
$$\mathcal L_n^{1-qe}=\mathcal L_n^\wedge\cup\{P_{q,\sigma}:q<\omega,\ \sigma\in\Sigma_n^{+}(q)\},$$
and let $\mathfrak T^{n,1-qe}_\alpha$ be the resulting expansion of $\mathfrak T_\alpha^n$. For later use, $\varphi_\sigma(x,y)$ denotes the first-order $\mathcal L_n^{\mathrm{rel}}$-formula obtained from $\sigma$ by restricting every quantifier to $I(x,y)$, replacing the order by $\subsetneq$, and replacing $C_i(z)$ by $\exists u(E_i(z,u)\wedge u\subseteq y)$.

\begin{lemma}[Path calculus]\label{lem:path-calculus}
    \begin{enumerate}
        \item\label{itm:path-concat} The characteristic of a concatenation of two transition words is effectively determined by their characteristics.
        \item\label{itm:path-split} If $q\geq1$, $\mathfrak U\equiv_{q+1}\mathfrak V$, and $\mathfrak U=\mathfrak U_0+\mathfrak U_1$ with both parts nonempty and $\mathfrak U_1$ having a least point, then $\mathfrak V=\mathfrak V_0+\mathfrak V_1$ for nonempty parts satisfying $\mathfrak U_i\equiv_q\mathfrak V_i$.
        \item\label{itm:path-effective} For every $q<\omega$, the set $\Sigma_n^{+}(q)$ and its addition table are effectively computable. Every member has a nonempty representative of order type below $\omega^\omega$; at rank $0$ we choose a one-letter representative.
    \end{enumerate}
\end{lemma}
\begin{proof}
    Item~(\ref{itm:path-concat}) follows from Lemma~\ref{lem:effective-characteristic-algebra}, and item~(\ref{itm:path-split}) follows from Lemma~\ref{lem:colored-facts}(\ref{itm:colored-splitting}). Item~(\ref{itm:path-effective}) follows from Theorem~\ref{thm:exact-algorithm} together with the effective characteristic algebra; every rational term has order type below $\omega^\omega$. For $q=0$ we replace the possibly empty least representative by any one-letter word, which has the same rank-$0$ characteristic.
\end{proof}

Fix variables $\bar x=(x_1,\ldots,x_k)$. Their \emph{meet closure} is the finite tree generated by $\epsilon$, the $x_i$, and their iterated meets.

We consider the following canonical trees, where $\theta<\omega^\omega$:
\begin{itemize}
    \item Type \MakeUppercase{\romannumeral 1}: $\mathfrak{T}_{\omega^\omega}^n$;
    \item Type \MakeUppercase{\romannumeral 2}-$\theta$: $\mathfrak{T}_{\theta}^n$;
    \item Type \MakeUppercase{\romannumeral 3}-$\theta$: $\mathfrak{T}_{\omega^\omega+\theta}^n$;
    \item Type \MakeUppercase{\romannumeral 4}-$\theta$: $\mathfrak{T}_{\omega_1+\theta}^n$.
\end{itemize}

Every type of canonical trees will have their own computably enumerable axiomatization and quantifier elimination theorem in somewhat different definable extension of $\mathcal{L}_n^{1-qe}$. The techniques we use are the same in essentials while differing in minor points. We may first deal with Type \MakeUppercase{\romannumeral 1} trees, because it is the simplest one and needs less considerations.

\subsection{Type \MakeUppercase{\romannumeral 1} canonical tree}\label{sec:type1syn}

We give below an axiomatization of the theory $T_n^1$ of the Type~\MakeUppercase{\romannumeral 1} canonical tree in the language $\mathcal L_n^{1-qe}$. Every displayed axiom with free variables is understood to be universally closed and therefore to denote a sentence. We first define the auxiliary theory $T_n^0$.
\begin{enumerate}[label=\textnormal{(A\arabic*)}]
    \item\label{ax:tree-order} $\subseteq$ is a partial order with least element $\epsilon$, it is downward linear, and $x\wedge y$ is the greatest common predecessor of $x$ and $y$. The axioms describing these assertions are rather simple and easy to list out in details, so we may omit them.

    \item\label{ax:named-successors} Each $E_i$ is the graph of a total function giving an immediate successor:
    $$\forall x\,\exists!y\,E_i(x,y),\qquad E_i(x,y)\longrightarrow x\subsetneq y\wedge\neg\exists z\,(x\subsetneq z\subsetneq y).$$
    If $i\neq j$, then $E_i(x,y)\wedge E_j(x,z)$ implies $y\neq z$.

    \item\label{ax:successor-cones} Every proper extension enters a named successor cone:
    $$x\subsetneq y\longrightarrow\bigvee_{i<n}\exists z\,\bigl(E_i(x,z)\wedge z\subseteq y\bigr).$$
    Consequently, if $x\subsetneq y$, then exactly one $i<n$ satisfies
    $C_i(x;y)$.  Existence follows from Axiom~\ref{ax:successor-cones}.  For uniqueness, two differently labelled immediate successors below $y$ would be comparable by downward linearity, contradicting immediacy and their distinctness.

    \item\label{ax:path-definitions} For every $q<\omega$ and
    $\sigma\in\Sigma_n^{+}(q)$,
    $$P_{q,\sigma}(x,y)\longleftrightarrow x\subsetneq y\wedge\varphi_\sigma(x,y),$$
    with $\varphi_\sigma$ as in the
    \hyperref[def:path-predicates]{definition of path predicates and relativization}.

    \item\label{ax:characteristic-exhaustion} For every $q<\omega$,
    $$x\subsetneq y\longrightarrow\bigvee_{\sigma\in\Sigma_n^{+}(q)}P_{q,\sigma}(x,y).$$
\end{enumerate}
To form $T_n^1$, add the following axiom scheme to $T_n^0$: \label{ax:characteristic-realization} For every $q<\omega$ and $\sigma\in\Sigma_n^{+}(q)$,
$$\forall x\,\exists y\,P_{q,\sigma}(x,y).$$
As one may immediately see, this axiom will play a central role in making this axiomatized theory complete and proving other desired theorems.
It is easy to observe that $\mathfrak{T}_{\omega^\omega}^{n,1,qe}\vDash T_n^1$ from the fact that every $\sigma\in\Sigma_n^{+}(q)$ could be satisfied on some ordinal $\leq g_{\mathrm{ord}}(n,q)<\omega^\omega$ and $\omega^\omega$ is additively indecomposable.

Now our first goal is to deduce the quantifier elimination. The idea is that, for (meet-closed) finitely many nodes in models of $T_n^1$, the only information we must record is the finite tree structure restricted on them and the path characteristics satisfied between every two of them. To refer to all these things together, we may define the concept of \MakeUppercase{\romannumeral 1}-Diagrams.

A \emph{type-\MakeUppercase{\romannumeral 1} level-$N$ diagram} $\Delta$ (for $N\ge 2$) consists of the following data:
\begin{itemize}
    \item the isomorphism type of the finite meet tree on these nodes; equivalently, the set of all meet terms together with the relations $\subseteq$ and $\wedge$ restricted to them;
    \item for every edge $(u,v)$ of the meet closure (i.e., $u\subsetneq v$ and there is no meet term strictly between $u$ and $v$), a distinguished characteristic $\sigma_{u,v}\in\Sigma_n^{+}(N)$.
\end{itemize}
These data satisfy the coherence condition that for any node $u$, the first colors of the characteristics $\sigma_{u,v}$ as $v$ ranges over all immediate successors of $u$ in the meet closure are pairwise distinct. (The first color of a path characteristic is the unique index $i<n$ such that the path contains a point immediately succeeding $u$ with color $i$; it is read off from $\sigma_{u,v}$.) We will simply name $\Delta$ as a $N$ \MakeUppercase{\romannumeral 1}-diagram.

Associated to each such \MakeUppercase{\romannumeral 1}-diagram $\Delta$ we define a quantifier-free formula $\delta_{\Delta}(\bar{x})$ of $\mathcal{L}_n^{1-qe}$, which expresses that the tuple $\bar{x}$ realizes the configuration described by $\Delta$ up to level-$N$ characteristics. Concretely, $\delta_{\Delta}(\bar{x})$ is the conjunction of the following:
\begin{enumerate}[label=(\roman*)]
    \item all equalities and inequalities among the meet terms of $\bar{x}$ (including the constant $\epsilon$);
    \item all order relations $u\subseteq v$ and $u\not\subseteq v$ for any two such terms;
    \item for every edge $(u,v)$ of the diagram, the atom $P_{N,\sigma_{u,v}}(u,v)$;
    \item for every edge $(u,v)$ and every $\sigma\in\Sigma_n^{+}(N)$ with $\sigma\neq\sigma_{u,v}$, the negated atom $\neg P_{N,\sigma}(u,v)$.
\end{enumerate}
(If the meet closure contains no proper edges, then $\delta_{\Delta}$ is simply the empty conjunction $\top$, with $\epsilon$ being the only node.)

Since $\Sigma_n^{+}(N)$ is finite, for fixed $k$ and $N$ there are only finitely many level-$N$ diagrams. Moreover, in any model $\mathfrak{M}\models T_n^1$, for every assignment of the variables $\bar{x}$, the predicates $P_{N,\sigma}$ evaluated on the meet terms determine exactly one diagram $\Delta$, and $\mathfrak{M}\models\delta_{\Delta}(\bar{x})$. Indeed, Axiom~\ref{ax:characteristic-exhaustion} guarantees that for any $u\subsetneq v$ there is a unique $\sigma$ with $P_{N,\sigma}(u,v)$, and different $\sigma$ are mutually exclusive because they correspond to logically inequivalent characteristics.

The organization of the finite meet closure follows the condensation argument of \cite[Lemma~5.1]{Kellerman2026}; the edge labels in Lemma~\ref{lem:composition1} retain the additional information carried by the named successor cones. Now we establish the most important lemma related to the notion of diagrams.

\begin{lemma}[Composition lemma \MakeUppercase{\romannumeral 1}]\label{lem:composition1}
    Let $\mathfrak A,\mathfrak B\models T_n^1$, let $\bar a\in A^k$ and $\bar b\in B^k$, and let $q\geq 0$. If $\bar a$ and $\bar b$ realize the same $(q+2)$ \MakeUppercase{\romannumeral 1}-diagram, then they satisfy the same $\mathcal L_n^\wedge$-formulae of quantifier rank at most $q$.
\end{lemma}
\begin{proof}
    $\exists$ maintains finite meet closures $X\subseteq A$ and $Y\subseteq B$ and an isomorphism $h:X\to Y$ of rooted meet-trees which maps the played tuple to its mate. With $r$ rounds remaining, the invariant requires corresponding edges to carry the same rank-$(r+2)$ characteristic. It holds initially with $r=q$ by the hypothesis.

    For the induction step, suppose that $r+1$ rounds remain. Thus corresponding old edges carry the same rank-$m$ characteristic, where $m=r+3$. Suppose $\forall$ plays $c\in A$; a play in $B$ is handled symmetrically. If $c\in X$, $\exists$ plays $h(c)$, and restricting the old rank-$m$ characteristics to rank $m-1=r+2$ gives the invariant after the move. Otherwise, let
    $$p=\max_{\subseteq}\{c\wedge a:a\in X\}.$$
    This maximum exists because the displayed set is finite and consists of predecessors of $c$, hence is linearly ordered. Moreover $c\wedge a=p\wedge a$ for every $a\in X$.

    First suppose $p\in X$, and put $p'=h(p)$. Let $i$ be the first color of the transition word from $p$ to $c$. No edge $(p,v)$ of $X$ begins with color $i$: otherwise $c\wedge v$ would properly extend $p$, contrary to maximality. By Axiom~\ref{ax:characteristic-exhaustion}, there is a unique $\sigma\in\Sigma_n^{+}(m-1)$ with $P_{m-1,\sigma}(p,c)$. Axiom \MakeUppercase{\romannumeral 1} gives $c'$ with $P_{m-1,\sigma}(p',c')$. Since $m-1=r+2\geq2$, this characteristic determines the first color, so $c'$ lies in the corresponding unused cone.  Extending $h$ by $c\mapsto c'$ therefore gives isomorphic meet closures; the new edge has characteristic $\sigma$ on both sides.

    Now suppose $p\notin X$. Let $v$ be the least member of $X$ strictly above $p$, and let $u$ be the greatest member of $X$ strictly below $p$. Then $(u,v)$ is an edge of $X$ and $u\subsetneq p\subsetneq v$. Write $u'=h(u)$ and $v'=h(v)$. Since $\mathfrak I(u,v)\equiv_m\mathfrak I(u',v')$, and since both parts are nonempty and $I(p,v)$ has least point $p$, the splitting part of Lemma~\ref{lem:path-calculus}(\ref{itm:path-split}) (with $m-1\geq2$) supplies $u'\subsetneq p'\subsetneq v'$ such that
    $$\mathfrak I(u,p)\equiv_{m-1}\mathfrak I(u',p'),\qquad\mathfrak I(p,v)\equiv_{m-1}\mathfrak I(p',v').$$
    If $c=p$, $\exists$ plays $p'$. Otherwise, the first color from $p$ to $c$ differs from the first color from $p$ to $v$, since equal colors would make $c\wedge v$ properly extend $p$. Apply Axiom~\ref{ax:characteristic-exhaustion} to the pair $(p,c)$ and then Axiom \MakeUppercase{\romannumeral 1} at $p'$ to obtain $c'$ with the same rank-$(m-1)$ characteristic. Level $m-1\geq2$ fixes its first color, which differs from that of $\mathfrak I(p',v')$; hence $c'\wedge v'=p'$, and all new meets correspond. Extend $h$ by $p\mapsto p'$ and, when $c\neq p$, by $c\mapsto c'$.

    In both cases corresponding old edges carry the same rank-$(m-1)$ characteristic, as do corresponding new and split edges. Since $m-1=r+2$, after $\exists$'s reply $r$ rounds remain and the invariant is restored. At the end, the rooted meet-tree map preserves $\epsilon$, $\wedge$, $\subseteq$, and equality. Moreover, $E_i(s,t)$ holds exactly when $s\subsetneq t$ and $\mathfrak I(s,t)$ is the one-point word of color $i$, a property determined by its rank-$2$ characteristic. Thus the map on the played tuples is an $\mathcal L_n^\wedge$-partial isomorphism, and $\exists$ wins.
\end{proof}

Here comes our main results.

\begin{theorem}[Completeness \MakeUppercase{\romannumeral 1}]\label{thm:tree-completeness1}
    The theory $T_n^1$ is complete, thus
    $$T_n^1=\operatorname{Th}(\mathfrak T_{\omega^\omega}^{n,1-qe}).$$
    The set of its $\mathcal L_n^{\mathrm{rel}}$-consequences is $\operatorname{Th}(\mathfrak T_{\omega^\omega}^{n,\mathrm{rel}})$; under the graph/function definitional equivalence, this theory corresponds exactly to $\operatorname{Th}(\mathfrak T_{\omega^\omega}^n)$.
\end{theorem}
\begin{proof}
    Let $\mathfrak A,\mathfrak B\models T_n^1$. For every $q$, their empty tuples have the same $(q+2)$ \MakeUppercase{\romannumeral 1}-diagram, consisting only of the root and having no edges. Lemma~\ref{lem:composition1} gives $\mathfrak A\equiv_q\mathfrak B$ in $\mathcal L_n^\wedge$ for every $q$. Given an $\mathcal L_n^{1-qe}$-sentence, Axiom~\ref{ax:path-definitions} expands its finitely many path predicates into an $\mathcal L_n^\wedge$-sentence of some finite rank, so $\mathfrak A\equiv\mathfrak B$ in the expanded language.
\end{proof}

\begin{theorem}[Quantifier elimination and decidability \MakeUppercase{\romannumeral 1}]\label{thm:1-qe}
    The theory $T_n^1$ is decidable, and admits effective quantifier elimination. More precisely, if $\varphi(\bar x)$ is an $\mathcal L_n^\wedge$-formula of quantifier rank $q$, then it is equivalent modulo $T_n^1$ to
    $$\bigvee_{\Delta\in\mathcal D_\varphi}\delta_\Delta(\bar x),$$
    where $\mathcal D_\varphi$ is the finite set of $(q+2)$ \MakeUppercase{\romannumeral 1}-diagrams $\Delta$ for which
    $$T_n^1\vdash\exists\bar x\,(\delta_\Delta(\bar x)\wedge\varphi(\bar x)).$$
    For an arbitrary $\mathcal L_n^{1-qe}$-formula $\psi$, first expand all its $P_{q,\sigma}$-predicates by Axiom~\ref{ax:path-definitions}. If the resulting $\mathcal L_n^\wedge$-formula has quantifier rank $Q$, the corresponding quantifier-free normal form uses $(Q+2)$ \MakeUppercase{\romannumeral 1}-diagrams.
\end{theorem}
\begin{proof}
    Every tuple has a unique level-$(q+2)$ diagram. By Lemma~\ref{lem:composition1}, tuples in possibly different models which have the same diagram agree on $\varphi$. Hence every consistent diagram decides $\varphi$, and the displayed disjunction is equivalent to $\varphi$. By Lemma~\ref{lem:path-calculus}(\ref{itm:path-effective}), the relevant characteristic sentences are computable, so the schemes in Axioms~\ref{ax:path-definitions}--\ref{ax:characteristic-exhaustion} and the finite lists of diagrams are recursive. Thus $T_n^1$ is recursively axiomatized. Since it is complete (and consistent), simultaneous proof search for a sentence and its negation decides its theorem set and selects $\mathcal D_\varphi$.
\end{proof}

And, as a corollary, we get the following result.
\begin{theorem}[Standard trees of type \MakeUppercase{\romannumeral 1}]\label{thm:1-classify}
    A standard tree $\mathfrak{T}_\alpha^n$ is elementarily equivalent to the canonical tree $\mathfrak{T}_{\omega^\omega}^n$ if and only if $(\alpha,<)$ is elementarily equivalent to $\omega^\omega$.
\end{theorem}
\begin{proof}
    We prove the two implications separately.

    In $\mathfrak{T}_\alpha^n$ the leftmost path is first‑order definable in $\mathcal L_{n,S}$ and, with the order $\subseteq$, is isomorphic to $(\alpha,<)$. Hence every sentence in the language of linear orders can be translated into an $\mathcal L_{n,S}$-sentence by relativizing its quantifiers to the definable leftmost path. From $\mathfrak{T}_\alpha^n\equiv\mathfrak{T}_{\omega^\omega}^n$ we obtain $(\alpha,<)\equiv(\omega^\omega,<)$.

    Assume $(\alpha,<)\equiv(\omega^\omega,<)$. By the Mostowski–Tarski analysis of the first‑order theory of well‑orderings (see \S\ref{sec:preliminaries}), $\alpha\equiv\omega^\omega$ implies that either $\alpha=\omega^\omega$ or $\alpha>\omega^\omega$ and the Cantor normal form of $\alpha$ contains no summand of the form $\omega^k$ with $k<\omega$; in particular, all exponents appearing in the normal form are at least $\omega$. Consequently, for every $\beta<\alpha$ and every $\gamma<\omega^\omega$ we have $\beta+\gamma<\alpha$.

    Now expand $\mathfrak{T}_\alpha^n$ to the language $\mathcal L_n^{1-qe}$ (see \S\ref{sec:type1syn}) by interpreting the predicates $P_{q,\sigma}$ as in the definition of path characteristics. We verify that this expansion satisfies the theory $T_n^1$, and the result follows easily.
\end{proof}

\subsection{Type \MakeUppercase{\romannumeral 2} canonical trees}
\label{sec:type2syn}

Recall that the \emph{Type~{\MakeUppercase{\romannumeral 2}}-$\theta$ canonical tree} for a limit ordinal $\theta<\omega^\omega$ is the structure $\mathfrak{T}_{\theta}^n$.   Every such $\theta$ is definable by a first‑order formula in the language of linear orders (Mostowski–Tarski, cf.\ \S\ref{subsec:babyprelim}), hence the property ``the height of a node is at least $\theta$'' is expressible in the language of trees. For a smooth quantifier‑elimination we introduce explicit predicates that record the \emph{residual height} of a node.

Write $\theta$ in expanded Cantor normal form $\omega^{\theta_1}+\dots+\omega^{\theta_m}$ with $\theta_1\ge\cdots\ge\theta_m\ge 1$.  The possible residual heights of nodes in $\mathfrak{T}_\theta^n$ are exactly the tail sums
$$R_\theta=\bigl\{\,\omega^{\theta_i}+\dots+\omega^{\theta_m}:1\le i\le m\,\bigr\}.$$
Equivalently, $R_\theta=\{\,\mu>0:\exists\gamma<\theta\;(\theta=\gamma+\mu)\,\}$; this is a finite set, so we may enumerate its elements as $\theta=\mu_1>\dots>\mu_{m}$. For every $\mu_i\in R_\theta$ we introduce a unary predicate $H_i$ with the intended interpretation
$$H_i(x)\;\Longleftrightarrow\; \operatorname{ht}(x)+\mu_i=\theta.$$
Let
$$\mathcal{L}_{n,\theta}^{2-qe}= \mathcal{L}_n^{1-qe}\cup\{H_1,\dots,H_{m}\}.$$

The theory $T_{n,\theta}^2$ in the language $\mathcal{L}_{n,\theta}^{2-qe}$ extends the base theory $T_n^0$ (see Section~\ref{sec:type1syn}) together with the following height axioms.

\begin{enumerate}[label=\textnormal{(B\arabic*)}]
    \item\label{ax:2H-partition}
        $$\forall x\Bigl(\bigvee_{i}H_i(x)\Bigr),\qquad \forall x(H_i(x)\to\neg H_j(x))(i\neq j).$$
    \item\label{ax:2H-root}
        $$H_1(\epsilon).$$
    \item\label{ax:2H-edge}
        If $x\subsetneq y$, $H_i(x)$, and $H_j(y)$, then $i\leq j$. Moreover, for every $q<\omega$ and $\sigma\in\Sigma_n^+(q)$, if $P_{q,\sigma}(x,y)$, then there is an ordinal $\lambda$ such that $\sigma$ is realizable on a coloring of $\lambda$ and $\lambda+\mu_j=\mu_i$.
    \item\label{ax:2H-realizability}
        Let $q<\omega$, $i,j\in\{1,\dots,m\}$.

If $\sigma\in\Sigma_n^+(q)$ is realizable on a colored ordinal of some order type $\alpha>0$ satisfying $\mu_i=\alpha+\mu_j$, then
        $$\forall x\,\bigl(H_i(x)\to\exists y\,(P_{q,\sigma}(x,y)\land H_j(y))\bigr).$$

If no positive $\alpha$ satisfies both $\mu_i=\alpha+\mu_j$ and realizability of $\sigma$ at $\alpha$, then
$$\forall x\,\forall y\bigl((H_i(x)\land H_j(y))\to\neg P_{q,\sigma}(x,y)\bigr).$$
\end{enumerate}
The quantification over $\lambda$ in Axiom~\ref{ax:2H-edge} is metatheoretic. For fixed $\mu_i$ and $\mu_j$, however, the set of ordinals $\lambda$ satisfying $\lambda+\mu_j=\mu_i$ is effectively describable below $\omega^\omega$. Theorem~\ref{lem:interval-realizabiliy} therefore makes this an effective axiom scheme.

As before, here comes the definition of \emph{type-{\MakeUppercase{\romannumeral 2}}-$\theta$ level-$N$ diagrams}. Fix $\theta<\omega^\omega$ and its residual height set $R_\theta=\{\mu_1>\mu_2>\dots>\mu_m\}$. A $\Delta$ (for $N\ge 2$) consists of:
\begin{itemize}
    \item the isomorphism type of a finite meet tree on a set of nodes, together with a labelling $\lambda\colon \text{nodes}\to\{1,\dots,m\}$ such that
    \begin{itemize}
        \item the root receives label $1$ (i.e.\ $\mu_1=\theta$);
        \item if $u\subsetneq v$ are nodes, then $\lambda(u)\leq\lambda(v)$ (so that $\mu_{\lambda(u)}\geq\mu_{\lambda(v)}$);
    \end{itemize}
    \item for every edge $(u,v)$ of the meet closure (i.e.\ $u\subsetneq v$ and no labelled node lies strictly between them), a distinguished characteristic $\sigma_{u,v}\in\Sigma_n^{+}(N)$.
\end{itemize}
The data satisfy the coherence conditions:
\begin{enumerate}
    \item[(i)] for any node $u$, the first colors of the characteristics $\sigma_{u,v}$ for the immediate successors $v$ of $u$ are pairwise distinct;
    \item[(ii)] for every edge $(u,v)$ with $\lambda(u)=i$, $\lambda(v)=j$, there exists an ordinal $\lambda$ such that $\lambda+\mu_j=\mu_i$ and $\sigma_{u,v}$ is realizable on a colored ordinal of length $\lambda$ (this is still decidable by Theorem~\ref{lem:interval-realizabiliy}).
\end{enumerate}
Associated with such a diagram $\Delta$ is a quantifier‑free formula $\delta_\Delta(\bar x)$ of $\mathcal L_{n,\theta}^{2-qe}$, expressing that the tuple $\bar x$ realizes $\Delta$ up to level‑$N$ characteristics.  Concretely, $\delta_\Delta(\bar x)$ is the conjunction of:
\begin{itemize}
    \item all equalities and inequalities among the meet terms of $\bar x$ (including $\epsilon$);
    \item all order relations $u\subseteq v$ and $u\not\subseteq v$ between these terms;
    \item for each meet term $t$, the atom $H_{\lambda(t)}(t)$;
    \item for every edge $(u,v)$ of the diagram, the atom $P_{N,\sigma_{u,v}}(u,v)$;
    \item for every edge $(u,v)$ and every $\sigma\in\Sigma_n^{+}(N)$ with $\sigma\neq\sigma_{u,v}$, the negated atom $\neg P_{N,\sigma}(u,v)$.
\end{itemize}
As before, for fixed $k$ and $N$ there are only finitely many such diagrams, and in any model of $T_{n,\theta}^2$ every tuple satisfies exactly one $\delta_\Delta$.
We first note that, for each $i$, the condition $\operatorname{ht}(x)+\mu_i=\theta$ is defined by a fixed $\mathcal L_n^\wedge$-formula in every model of $T_{n,\theta}^2$, not only in the standard tree. Indeed, since $\theta<\omega^\omega$, a sufficiently complex formula can successively record the numbers of limit points, limit points of limit points, and so on along the path from $\epsilon$ to $x$; the expanded Cantor normal form of $\theta$ makes this a finite process. Thus $H_i$ is considered as a shortcut for this formula.

Then there is the associated composition lemma.
\begin{lemma}[Composition lemma \MakeUppercase{\romannumeral 2}]\label{lem:composition2}
    Let $\mathfrak A,\mathfrak B\models T_{n,\theta}^2$, let $\bar a\in A^k$, $\bar b\in B^k$, and let $q\ge 0$. If $\bar a$ and $\bar b$ realize the same $(q+t)$ type-\MakeUppercase{\romannumeral 2}-$\theta$ diagram, where $t$ is a sufficiently large padding, then they satisfy the same $\mathcal L_n^\wedge\cup\{H_1,\dots,H_m\}$-formulae of quantifier rank at most $q$.
\end{lemma}
\begin{proof}
    The proof parallels that of Lemma~\ref{lem:composition1}, with the additional task of preserving the height labels. We outline the main modifications.

    $\exists$ maintains a finite meet closure $X\subseteq A$, $Y\subseteq B$, an isomorphism $h\colon X\to Y$ of labelled meet trees, and the invariant that with $r$ rounds remaining, corresponding edges carry the same rank-$(r+t)$ characteristic. Initially the hypothesis provides this with $r=q$.

    If $\forall$ plays $c\in A$, let $p=\max_{\subseteq}\{c\wedge a : a\in X\}$. If $p\in X$, let $p'=h(p)$ and let $H_i$ be the height label of $p$ (hence also of $p'$). Let $\sigma$ be the unique rank-$(r+t)$ characteristic with $P_{r+t,\sigma}(p,c)$; its first color differs from those of the edges in $X$ starting at $p$. Let $H_j$ be the height label of $c$ in $\mathfrak A$. From Axiom~\ref{ax:2H-edge} we obtain that $\sigma$ is realizable on an ordinal of length $\lambda$ with $\lambda+\mu_j=\mu_i$. By Axiom~\ref{ax:2H-realizability} applied to $p'$, there exists $c'\in B$ with $H_j(c')$ and $P_{r+t,\sigma}(p',c')$. $\exists$ plays $c'$, and the map $c\mapsto c'$ extends $h$ to an isomorphism of the enlarged meet closures; the new edge carries the same characteristic $\sigma$ on both sides.

    Now suppose $p\notin X$. Let $v$ be the least element of $X$ strictly above $p$, and $u$ the greatest element of $X$ strictly below $p$; then $(u,v)$ is an edge of $X$. Write $u'=h(u)$ and $v'=h(v)$. Let $H_a,H_b$ be the labels of $u,v$, respectively, and let $\sigma_{u,v}$ be the edge characteristic of rank $m=r+t$. Since $\mathfrak I(u,v)\equiv_m\mathfrak I(u',v')$ and $I(p,v)$ has a least point, Lemma~\ref{lem:path-calculus}(\ref{itm:path-split}) yields $p'\in B$ with $u'\subsetneq p'\subsetneq v'$ such that
    $$\mathfrak I(u,p)\equiv_{m-1}\mathfrak I(u',p'),\qquad \mathfrak I(p,v)\equiv_{m-1}\mathfrak I(p',v').$$
    In $\mathfrak A$, the node $p$ has some height label $H_c$.  The lengths of the left and right parts satisfy $\operatorname{len}(u,p)+\operatorname{len}(p,v)=\lambda$ with $\lambda+\mu_b=\mu_a$, and $\operatorname{len}(p,v)+\mu_b=\mu_c$.  Hence $\operatorname{len}(u,p)+\mu_c=\mu_a$.  The realized characteristics $\sigma_{u,p}$ and $\sigma_{p,v}$ (of rank $m-1$) together with the height labels satisfy the realizability conditions expressed in Axioms~\ref{ax:2H-edge} and ~\ref{ax:2H-realizability}.  Because the same combinatorial situation occurs in $\mathfrak B$ with the isomorphic edge $(u',v')$ and the same rank-$(m-1)$ characteristics, there exists $p'$ in $\mathfrak B$ with the same height label $H_c$ and with the same left and right characteristics.  $\exists$ plays this $p'$ as the response for $p$.  If $c\neq p$, then $c$ is above $p$ and not in $X$; we then apply the argument of the first case inside the cone above $p$ (now $p$ is in the closure).  The new elements $p',c'$ receive the same height labels as their counterparts, and the invariant is restored with $r-1$ rounds remaining.

    After $q$ rounds the map on the played tuples respects $\epsilon$, $\wedge$, $\subseteq$, $E_i$, and all height predicates $H_j$.  Therefore it is a partial $\mathcal L_n^\wedge\cup\{H_1,\dots,H_m\}$-isomorphism, and $\exists$ wins the EF game. 
\end{proof}

Then we get completeness and decidability and quantifier elimination.

\begin{theorem}[Completeness \MakeUppercase{\romannumeral 2}]\label{thm:tree-completeness2}
    For every nonzero limit ordinal $\theta<\omega^\omega$, the theory $T_{n,\theta}^2$ is complete and coincides with the first‑order theory of the canonical structure $\mathfrak T_{\theta}^{n,2-qe}$. Consequently, its $\mathcal L_n^{\mathrm{rel}}$-reduct is $\operatorname{Th}(\mathfrak T_{\theta}^{n,\mathrm{rel}})$.
\end{theorem}
\begin{proof}
    The empty tuple in any model satisfies the same $(q+t)$-diagram for a sufficiently large padding $t$ (only the root with label $\mu_1=\theta$, no edges).  Lemma~\ref{lem:composition2} gives elementary equivalence up to any finite rank, hence full equivalence.
\end{proof}

\begin{theorem}[Quantifier elimination and decidability \MakeUppercase{\romannumeral 2}]\label{thm:2-qe}
    The theory $T_{n,\theta}^2$ is decidable and admits effective quantifier elimination.  Every $\mathcal L_n^\wedge\cup\{H_1,\dots,H_m\}$-formula $\varphi(\bar x)$ of quantifier rank $q$ is equivalent modulo $T_{n,\theta}^2$ to the disjunction 
    $$\bigvee_{\Delta\in\mathcal D_\varphi}\delta_\Delta(\bar x),$$
    where $\mathcal D_\varphi$ is the finite set of $(q+t)$ type-{\MakeUppercase{\romannumeral 2}}-$\theta$ diagrams consistent with $\varphi$, for a sufficiently large padding $t$.  The same holds for $\mathcal L_{n,\theta}^{2-qe}$-formulae after expanding the path predicates.
\end{theorem}
\begin{proof}
    Since every tuple determines a unique diagram, and tuples with the same diagram agree on all rank‑$q$ formulae by Lemma~\ref{lem:composition2}, the equivalence follows. The axiomatization is recursive: the finite sets $R_\theta$ and $\Sigma_n^+(q)$ can be computed, the realizability conditions in Axioms~\ref{ax:2H-edge} and ~\ref{ax:2H-realizability} are decidable by Theorem~\ref{lem:interval-realizabiliy}, and the base theory $T_n^0$ is recursive. Completeness (Theorem~\ref{thm:tree-completeness2}) then makes the theory decidable, and the set $\mathcal D_\varphi$ computable by exhaustive search for consistency.
\end{proof}

The following classification result follows easily.
\begin{theorem}[Standard trees of type \MakeUppercase{\romannumeral 2}]\label{thm:2-classify}
    A standard tree $\mathfrak{T}_\alpha^n$ is elementarily equivalent to the canonical tree $\mathfrak{T}_{\theta}^n$ if and only if $\alpha=\theta$.
\end{theorem}
\begin{proof}
    This follows because the definable leftmost paths of the two trees are elementarily equivalent. Since $\alpha,\theta<\omega^\omega$, the Mostowski--Tarski classification gives $\alpha=\theta$.
\end{proof}

\subsection{Type \MakeUppercase{\romannumeral 3} and type \MakeUppercase{\romannumeral 4} canonical trees}\label{sec:type34syn}

Now we take a look at type \MakeUppercase{\romannumeral 3} and type \MakeUppercase{\romannumeral 4} canonical trees. We may consider \MakeUppercase{\romannumeral 3} at first, then \MakeUppercase{\romannumeral 4} will only need a slight modification.

As in the Type~\MakeUppercase{\romannumeral 2} case, retain the residual-height predicates $H_1,\dots,H_m$. Place every node of height below $\omega^\omega$ in $H_0$. Define $\operatorname{SL}(x)=x$ for $H_0$-nodes; otherwise, define $\operatorname{SL}(x)$ to be the least $H_1$-predecessor of $x$.
$$y=\operatorname{SL}(x)\Longleftrightarrow\bigl((H_0(x)\land y=x)\lor(\neg H_0(x)\land y\subseteq x\land H_1(y)\land\forall z\,(z\subsetneq y\to H_0(z)))\bigr).$$
Now let
$$\mathcal{L}_{n,\theta}^{3-qe}=\mathcal{L}_{n,\theta}^{2-qe}\cup\{H_0,\operatorname{SL}\},$$
The theory $T_{n,\theta}^3$ extends $T_n^0$ and the Type~\MakeUppercase{\romannumeral 2} edge and realizability schemes, with $i,j\in\{1,\dots,m\}$. It also includes the following axioms:

\begin{enumerate}[label=\textnormal{(C\arabic*)}]
    \item\label{ax:3H-partition}
        $$\forall x\Bigl(\bigvee_{i}H_i(x)\Bigr),\qquad \forall x(H_i(x)\to\neg H_j(x))(i\neq j).$$
    \item\label{ax:3H-root}
        $$H_0(\epsilon).$$
    \item\label{ax:3H-downward}
        \(H_0\) is downward closed:
        $$\forall x,y\,\bigl(y\subsetneq x \land H_0(x)\to H_0(y)\bigr).$$
    \item\label{ax:3H-low}
        The \(H_0\)-part satisfies the Type~\MakeUppercase{\romannumeral 1} extension axiom relativised to \(H_0\): for every \(q<\omega\) and every \(\sigma\in\Sigma_n^{+}(q)\),
        $$\forall x\,\bigl(H_0(x)\to\exists y\,(H_0(y)\land P_{q,\sigma}(x,y))\bigr).$$
    \item\label{ax:3H-SL-def}
        The $\operatorname{SL}$ function satisfies our formal definition above.
    \item\label{ax:3H-SL-realizability}
        For each \(q<\omega\) and \(\sigma\in\Sigma_n^{+}(q)\),
        \begin{itemize}
            \item if \(\sigma\) can be realized on a colored ordinal of order type \(\omega^\omega\), then
            $$\forall x\,\bigl(H_0(x)\to\exists z\,(H_1(z)\land x\subsetneq z\land P_{q,\sigma}(x,\operatorname{SL}(z)))\bigr);$$
            \item otherwise,
            $$\forall x,z\,\bigl(H_0(x)\land H_1(z)\land x\subsetneq z\to\neg P_{q,\sigma}(x,\operatorname{SL}(z))\bigr).$$
        \end{itemize}
\end{enumerate}

We note that $H_0,H_1,\dots,H_m$ and $\operatorname{SL}$ are also definable in $\mathcal L_n^\wedge$ in every model of $T_{n,\theta}^3$, not only in the standard tree. The height predicates are defined by the same finite iteration of limit-point conditions as in Type~\MakeUppercase{\romannumeral 2}, and then $\operatorname{SL}$ is defined by Axiom~\ref{ax:3H-SL-def}.

From our discussions in Section~\ref{sec:realizability}, the axiom schema~\ref{ax:3H-SL-realizability} is enumerable.

For a tuple $\bar x=(x_1,\dots,x_k)$ in a model of $T_{n,\theta}^3$, we form its \emph{special meet closure} $X$ as follows. Start with $X_0$, the ordinary meet closure of $\{\epsilon, x_1,\dots,x_k\}$. Then let $X$ be the closure of $X_0$ under the function $\operatorname{SL}$. Because $\operatorname{SL}(\operatorname{SL}(u))=\operatorname{SL}(u)$ for every $u$ (a consequence of the axioms on $\operatorname{SL}$), the set $X$ remains finite and consists of the terms built from the variables using $\wedge$, $\epsilon$ and $\operatorname{SL}$. Moreover, $X$ is still closed under meet, since the meet of any two such terms is again a term (the meet of $\operatorname{SL}(u)$ and $\operatorname{SL}(v)$ can be expressed using the meet of $u$ and $v$ and possibly $\operatorname{SL}$). We now define the appropriate diagrams for this type.

Fix $\theta<\omega^\omega$, its residual heights $R_\theta=\{\mu_1>\dots>\mu_m\}$ with $\mu_1=\theta$, and a natural number $N\ge 2$. A \emph{type~{\MakeUppercase{\romannumeral 3}}-$\theta$ level-$N$ diagram} $\Delta$ consists of the following data:
\begin{itemize}
    \item A finite meet tree $T$ whose nodes are the elements of the special meet closure of a distinguished tuple of \emph{generators} (corresponding to the variables). The tree is equipped with a labeling $\lambda: T\to\{0,1,\dots,m\}$ and a unary function $s: T\to T$.
    \item The labeling satisfies:
    \begin{enumerate}[label=(\roman*)]
        \item $\lambda(\epsilon)=0$;
        \item if $u\subsetneq v$, then $\lambda(u)\le\lambda(v)$.
    \end{enumerate}
    \item For every edge $(u,v)$ of $T$ (i.e. $u\subsetneq v$ and no node of $T$ lies strictly between them) a distinguished characteristic $\sigma_{u,v}\in\Sigma_n^{+}(N)$.
\end{itemize}
These data must obey the following coherence conditions:
\begin{enumerate}[label=(\roman*)]
    \item For any node $u$, the first colors of the characteristics on the edges emanating upward from $u$ are pairwise distinct.
    \item If $\lambda(u)=i$, $\lambda(v)=j$ with $i,j\ge 1$, then there exists an ordinal $\lambda$ such that $\lambda+\mu_j=\mu_i$ and $\sigma_{u,v}$ is realizable on a colored ordinal of length $\lambda$ (this is decidable by Theorem~\ref{lem:interval-realizabiliy}).
\end{enumerate}

Every such diagram $\Delta$ gives rise to a quantifier-free $\mathcal{L}_{n,\theta}^{3-qe}$-formula $\delta_\Delta(\bar x)$ in the variables that generate $T$. It is the conjunction of:
\begin{enumerate}[label=(\alph*)]
    \item all equalities and inequalities among the terms of the special meet closure generated by $\bar x$ (including equalities of the form $\operatorname{SL}(t)=t'$ that correspond to the function $s$ in the diagram);
    \item the order relations $u\subseteq v$ and $u\not\subseteq v$ between those terms;
    \item for each term $t$, the atom $H_{\lambda(t)}(t)$;
    \item for every edge $(u,v)$ of $T$, the atom $P_{N,\sigma_{u,v}}(u,v)$.
\end{enumerate}
By construction, for any tuple $\bar a$ in a model of $T_{n,\theta}^{3}$, its special meet closure together with the actual values of the $\operatorname{SL}$ function determines exactly one diagram $\Delta$, and the model satisfies $\delta_\Delta(\bar a)$. There are finitely many choices for the parameters involved. Thus, for fixed $k$ and $N$, there are only finitely many level-$N$ diagrams.

We now establish the composition lemma.

\begin{lemma}[Composition lemma \MakeUppercase{\romannumeral 3}]\label{lem:composition3}
    Let $\mathfrak A,\mathfrak B\models T_{n,\theta}^3$, let $\bar a\in A^k$ and $\bar b\in B^k$, and let $q\ge 0$. If $\bar a$ and $\bar b$ realize the same $(q+t)$ type-\MakeUppercase{\romannumeral 3}-$\theta$ diagram, where $t$ is a sufficiently large padding, then they satisfy the same $\mathcal L_n^\wedge\cup\{H_0,H_1,\dots,H_m,\operatorname{SL}\}$-formulae of quantifier rank at most $q$.
\end{lemma}
\begin{proof}
    $\exists$ maintains a finite set $X\subseteq A$ and an isomorphic copy $Y\subseteq B$ together with an isomorphism $h: X\to Y$ of the structures induced by the language (including $\operatorname{SL}$, the order, the height predicates, and the edge characteristics). Initially $X$ and $Y$ are the special meet closures of $\bar a$ and $\bar b$, and $h$ is the map induced by the common diagram. With $r$ rounds remaining, the invariant demands that corresponding edges carry the same rank-$(r+t)$ characteristic.

    When $\forall$ plays a new element $c\in A$, let $p$ be the $\subseteq$-maximum of $\{c\wedge a:a\in X\}$, where $X$ is the current special meet closure. After choosing the response to $c$, extend both sides to their special meet closures. Two main cases arise.

    \textbf{Case 1: $p\in X$.} Let $i$ be the height label of $p$. The path from $p$ to $c$ starts with a color that, by maximality of $p$, differs from the first colors of all edges in $X$ going upward from $p$. The characteristic $\sigma$ of $I(p,c)$ at rank $r+t$ is determined. If $i=0$, then Axiom~\ref{ax:3H-low} (or \ref{ax:3H-SL-realizability} if the path crosses the special layer) guaranties the existence of a suitable $c'\in B$ with the same characteristic and the correct height labels. If $i\ge 1$, Axiom~\ref{ax:2H-realizability} provides a $c'$ in $B$ because the residual height constraints match. In either subcase, the $\operatorname{SL}$-images of the new nodes are computed according to the diagram and coincide with the values forced by the axioms.

    \textbf{Case 2: $p\notin X$.} Then $p$ lies strictly between two consecutive nodes $u,v$ of $X$ with $u\subsetneq v$. The path $I(u,v)$ splits at $p$. Since the edge characteristic of $(u,v)$ has rank $r+t$, Lemma~\ref{lem:path-calculus}(\ref{itm:path-split}) yields a splitting point $p'\in B$ with matching left and right characteristics, both of rank $r+t-1$. The height label of $p$ and its $\operatorname{SL}$-image must then be realized using the corresponding Type~III axioms. If $c\neq p$, apply the realizability argument of Case~1 above to the cone above $p$.

    In all situations, the extended map respects $\operatorname{SL}$, the height labels, and edge characteristics of rank $r+t-1=(r-1)+t$, so the induction continues with $r-1$ rounds remaining.
\end{proof}

With the composition lemma in hand, we obtain the expected metatheorems for Type~{\MakeUppercase{\romannumeral 3}}.

\begin{theorem}[Completeness \MakeUppercase{\romannumeral 3}]\label{thm:tree-completeness3}
    For every nonzero limit ordinal $\theta<\omega^\omega$, the theory $T_{n,\theta}^3$ is complete and coincides with the first‑order theory of the canonical structure $\mathfrak T_{\omega^\omega+\theta}^{n,3-qe}$. Hence its $\mathcal L_n^{\mathrm{rel}}$-reduct is $\operatorname{Th}(\mathfrak T_{\omega^\omega+\theta}^{n,\mathrm{rel}})$.
\end{theorem}
\begin{proof}
    Let $\mathfrak A,\mathfrak B\models T_{n,\theta}^3$. For every $q$, their empty tuples have the same $(q+t)$-diagram for a sufficiently large padding $t$, consisting only of $\epsilon$ with label $0$ and no edges. Lemma~\ref{lem:composition3} gives $\mathfrak A\equiv_q\mathfrak B$ for every $q$, and hence $\mathfrak A\equiv\mathfrak B$.
\end{proof}

\begin{theorem}[Quantifier elimination and decidability \MakeUppercase{\romannumeral 3}]\label{thm:3-qe}
    The theory $T_{n,\theta}^3$ is decidable and admits effective quantifier elimination. Every $\mathcal L_n^\wedge\cup\{H_0,\dots,H_m,\operatorname{SL}\}$-formula $\varphi(\bar x)$ of quantifier rank $q$ is equivalent modulo $T_{n,\theta}^3$ to the disjunction
    $$\bigvee_{\Delta\in\mathcal D_\varphi}\delta_\Delta(\bar x),$$
    where $\mathcal D_\varphi$ is the finite set of $(q+t)$ type-\MakeUppercase{\romannumeral 3}-$\theta$ diagrams consistent with $\varphi$, for a sufficiently large padding $t$.
\end{theorem}
\begin{proof}
    By Lemma~\ref{lem:composition3}, tuples with the same diagram satisfy the same rank-$q$ formulae. The axiomatization is recursively enumerable because the realizability conditions are decidable (Theorem~\ref{lem:interval-realizabiliy}), the set of characteristics is finite and computable, and the additional axioms for $H_0$ and $\operatorname{SL}$ are schematic but effectively presented. Completeness then guarantees decidability, and the set $\mathcal D_\varphi$ can be found by an exhaustive search for consistency.
\end{proof}

The following classification result follows from Theorem~\ref{lem:cofinal-stretching}.
\begin{theorem}[Standard trees of type \MakeUppercase{\romannumeral 3}]\label{thm:3-classify}
    A standard tree $\mathfrak{T}_\alpha^n$ is elementarily equivalent to the canonical tree $\mathfrak{T}_{\omega^\omega+\theta}^n$ if and only if $\alpha=\alpha_0+\theta$ where $\alpha_0\equiv\omega^\omega$ and $\cf(\alpha_0)=\omega$, and $\theta>0$.
\end{theorem}
\begin{proof}
    We prove the two implications separately.

    ($\Rightarrow$) Assume $\mathfrak{T}_\alpha^n\equiv\mathfrak{T}_{\omega^\omega+\theta}^n$.  The left‑most path is definable in $\mathcal{L}_{n,S}$ and, when ordered by $\subseteq$, is isomorphic to $(\alpha,<)$; similarly for $\omega^\omega+\theta$. Hence $(\alpha,<)\equiv(\omega^\omega+\theta,<)$. By the Mostowski–Tarski analysis of the first‑order theory of well‑orders, every ordinal elementarily equivalent to $\omega^\omega+\theta$ can be uniquely written as $\alpha_0+\theta$, where $\alpha_0\equiv\omega^\omega$.

    It remains to prove that $\cf(\alpha_0)=\omega$. Choose a node at the special layer whose transition word contains an unbounded discrete set of $C_0$-positions, with a least $C_0$-position and an immediate $C_0$-predecessor for every other $C_0$-position. The corresponding first-order tree sentence holds in $\mathfrak T_{\omega^\omega+\theta}^n$ and therefore in $\mathfrak T_\alpha^n$. Its realization yields a countable cofinal sequence in $\alpha_0$, so $\cf(\alpha_0)=\omega$.

    ($\Leftarrow$) Suppose $\alpha=\alpha_0+\theta$, $\alpha_0\equiv\omega^\omega$ and $\cf(\alpha_0)=\omega$. We expand $\mathfrak{T}_\alpha^n$ to the language $\mathcal{L}_{n,\theta}^{3-qe}$ by interpreting the new predicates as follows:
    \begin{itemize}
        \item $H_0(x)$ iff $\operatorname{ht}(x)<\alpha_0$;
        \item $H_i(x)$ ($i=1,\dots,m$) iff $\operatorname{ht}(x)+\mu_i=\alpha$ (where $\mu_i\in R_\theta$ are the residual heights of $\theta$);
        \item $\operatorname{SL}(x)$ is the unique node on the path from $\epsilon$ to $x$ of height $\alpha_0$ if $\operatorname{ht}(x)\ge\alpha_0$, and $\operatorname{SL}(x)=x$ otherwise.
    \end{itemize}
    We claim that this expansion satisfies all axioms of $T_{n,\theta}^3$.

    Axioms~\ref{ax:3H-partition}--\ref{ax:3H-downward} are immediate from the definition. For Axiom~\ref{ax:3H-low}, fix $q<\omega$, $\sigma\in\Sigma_n^{+}(q)$, and $x$ with $H_0(x)$.  Because $\alpha_0\equiv\omega^\omega$ and $\cf(\alpha_0)=\omega$, the ordinal $\alpha_0$ is $q$-universal $\omega$-cofinal (Lemma~\ref{lem:universal-large}). In particular, the tail of the tree above $x$ inside the $H_0$‑part has height $\alpha_0-\operatorname{ht}(x)$, which is again an ordinal $\equiv\omega^\omega$ with cofinality $\omega$.  By the cofinal stretching argument (Theorem~\ref{lem:cofinal-stretching}), every $q$-characteristic $\sigma$ that is realizable on a finitely colored ordinal (all of them are realizable on some ordinal $<\omega^\omega$) occurs as the characteristic of a path $I(x,y)$ with $H_0(y)$. Hence the relativised Type~{\MakeUppercase{\romannumeral 1}} axiom holds.

    Axiom~\ref{ax:3H-SL-def} holds by construction. For Axiom~\ref{ax:3H-SL-realizability}, fix $q<\omega$ and $\sigma\in\Sigma_n^+(q)$.
    \begin{itemize}
        \item If $\sigma$ is realizable on a colored ordinal of order type $\omega^\omega$, then by the cofinal stretching theorem applied to the initial segment $\alpha_0$, there exists for every $x$ in $H_0$ a node $z$ of height $\ge\alpha_0$ such that the path from $x$ to $\operatorname{SL}(z)$ (which lies exactly at height $\alpha_0$) has characteristic $\sigma$.  Indeed, $\alpha_0$ can be written as $\sum_{\xi<\omega}\beta_\xi$ where each $\beta_\xi\equiv\omega^\omega$; therefore, the interval between any $x$ and the special layer contains arbitrarily long blocks that are $q$-universal $\omega$-cofinal, allowing us to realize $\sigma$.
        \item If $\sigma$ is not realizable on $\omega^\omega$, then no path from a node in $H_0$ to the special layer can have characteristic $\sigma$, because the length of such a path would be an ordinal $<\alpha_0$ that is elementarily equivalent to an ordinal $<\omega^\omega$, and the first‑order properties of the path length would contradict the realizability of $\sigma$ on $\omega^\omega$.
    \end{itemize}
    Thus the axiom scheme is satisfied.

    The tail above height $\alpha_0$ has order type exactly $\theta$, and the residual height predicates $H_i$ are interpreted exactly as in the canonical tree $\mathfrak{T}_{\omega^\omega+\theta}^{n,3-qe}$.  Consequently, Axioms~\ref{ax:2H-edge} and~\ref{ax:2H-realizability} (restricted to $i,j\ge1$) hold verbatim, because the colored ordinals that can appear inside the $\theta$-tail are precisely those realizable on ordinals $<\theta$, and the height constraints match.

    Therefore the expanded structure $\mathfrak{T}_\alpha^{n,3-qe}$ is a model of $T_{n,\theta}^3$. By Theorem~\ref{thm:tree-completeness3}, the theory $T_{n,\theta}^3$ is complete, so
    $$\mathfrak{T}_\alpha^{n,3-qe}\equiv\mathfrak{T}_{\omega^\omega+\theta}^{n,3-qe}.$$
    Reducing to the common language $\mathcal{L}_{n,S}$ yields $\mathfrak{T}_\alpha^n\equiv\mathfrak{T}_{\omega^\omega+\theta}^n$.
\end{proof}

The language $\mathcal{L}_{n,\theta}^{4-qe}$ is exactly the same as $\mathcal{L}_{n,\theta}^{3-qe}$.  Define $T_{n,\theta}^4$ by replacing realizability on $\omega^\omega$ in Axiom~\ref{ax:3H-SL-realizability} with realizability on $\omega_1$. The corresponding Type~\MakeUppercase{\romannumeral 4} classification uses $\alpha_0\equiv\omega^\omega$ and $\cf(\alpha_0)>\omega$.  Completeness \MakeUppercase{\romannumeral 4} and Quantifier elimination and decidability \MakeUppercase{\romannumeral 4} are both proved in the same way.

\begin{theorem}\label{theorem:inequivalent}
  These four kinds of trees are pairwise elementarily inequivalent
\end{theorem}
\begin{proof}

  The reason is longer paths can code more information. Different values of $\theta$ are distinguished by the Mostowski--Tarski theorem ~\ref{theorem:DMT}.

  For a fixed $\theta$.

  (2) is not elementarily equivalent to (1), (3), or (4): choose $q$ sufficiently large that $\theta<\omega^{|\Sigma_{n-1}(q)|}$. By Lemma~\ref{lem:marker-hierarchy}, each of the latter three trees contains a node whose transition word has a characteristic order type at least $\omega^{|\Sigma_{n-1}(q)|}$.

  (1) is not elementarily equivalent to (3) or (4): in the latter trees there is a node of residual height $\theta$.

  (3) is not elementarily equivalent to (4). At the special layer of (3), one can choose a transition word such that its $C_0$-positions are cofinal and have a least element, with every other $C_0$-position having an immediate $C_0$-predecessor. This property forces the initial block to have countable cofinality.

\end{proof}

\subsection{The common first order theory of FnBTs}

As we have shown before, the first order theory of every standard tree is decidable. A natural question is to ask what is the common first-order theory $T_n$ of all standard trees. Since ultraproducts may contain nonstandard models, obviously the class of all standard trees is not an elementary class. However, we may wish a weaker result, once $q<\omega$ is fixed,  every model in the elementary class generated from the common theory is $q$-elementarily equivalent to some standard tree ; in fact, a similar relation between definably well-founded trees (which form an elementary class) and well-founded trees (which do not) is discussed in \cite{Doets1987}.

Directly, we may wish to take the intersection of the reduct of $T_n^1$ and all $T_{n,\theta}^{i}(i=2,3,4;\theta<\omega^\omega)$ in the language $\mathcal{L}^\wedge$. However, even though we know that they are all decidable theories, or that they are all $\Delta_1$, their countable intersection is just a $\Pi_1$ set. If we can describe the common theory $T_n$, then it would be computably enumerable, or that it would be $\Sigma_1$, then it is automatically decidable since $\Pi_1\cap\Sigma_1=\Delta_1$.

However, we can omit the hard work of writing down tons of axioms for $T_n$; instead, we notice that there are only finitely many sentences (up to equivalence) with quantifier rank $q$ in the language $\mathcal{L}^\wedge$, so there are only finitely many boolean combinations of all of them.
For each $q$, let $\chi_{q,1},\dots,\chi_{q,N_q}$ be the complete rank-$q$ theories realized by standard trees. The disjunction $\bigvee_{i\leq N_q}\chi_{q,i}$ holds in every standard tree, and every model satisfying it is $q$-elementarily equivalent to a standard tree. Taking these sentences for all $q$ gives an axiomatization of the intersection of the canonical theories.

Immediately $T_n$ is satisfied by every standard tree and is computably enumerable, and we have our claimed relation between models of $T_n$ and standard trees as above: $q$-equivalence comes from the fact that they choose the same disjunct. However, this relation shows that $T_n$ is actually the intersection of all canonical theories: every sentence that is consistent with $T_n$ will be satisfied by some standard tree. And our work is done.

\section{FnBTs and MSO structures over ordinals}

In the previous sections we obtained a complete classification of the elementary theories of the full $n$-branching ordinal trees $\mathfrak{T}_\alpha^n$ and established effective quantifier elimination and decidability for all canonical representatives.  In this final section we situate these structures within the broader logical landscape by relating them to monadic second‑order logic over ordinals and by settling the question of the definability of height inside the trees.

A node of $\mathfrak{T}_\alpha^n$ is a function $s:\beta\to n$ for some $\beta<\alpha$.  Such a function is naturally coded by an $n$‑tuple of pairwise disjoint subsets of $\beta$ whose union is $\beta$. Concretely, for $i<n$ put $X_i = \{\gamma<\beta\mid s(\gamma)=i\}.$ Then $(X_0,\dots,X_{n-1})$ is an $n$‑tuple of monadic predicates on $(\beta,<)$ that partitions an initial segment of $\alpha$.  Conversely, every such $n$‑tuple that partitions an initial segment of $\alpha$ arises from a unique node.

The tree order $\subseteq$ corresponds to the initial segment extension of the domains, which is easily expressed by monadic formulas. The successor functions $S_i$ correspond to adjoining one new ordinal at the end of the domain and placing it into $X_i$, again a monadically definable operation.  Hence there is a uniform interpretation of the $\mathcal{L}_{n,S}$-structure $\mathfrak{T}_\alpha^n$ in the MSO structure $(\alpha,<)$. In particular, the first‑order theory of $\mathfrak{T}_\alpha^n$ is many‑one reducible to the MSO theory of $(\alpha,<)$.

For the reverse interpretation we enrich the tree with a binary relation $x\sim y$ intended to mean ``$x$ and $y$ have the same height''.  Work inside $\mathfrak{T}_{\alpha+1}^n$; its nodes of height exactly $\alpha$ are precisely the functions $\alpha\to n$.  Fix $n\ge 2$ and let the \emph{zero path} be the constant‑$0$ node $\mathbf{0}_\alpha$.  Every subset $A\subseteq\alpha$ can be encoded by the node $c_A$ defined by
\[
c_A(\gamma) = \begin{cases}
1 & \text{if }\gamma\in A,\\
0 & \text{if }\gamma\notin A .
\end{cases}
\]
Thus the family $\{x\mid x\sim\mathbf{0}_\alpha\}$ is in bijection with the power set of $\alpha$.  The order $<$ on $\alpha$ can be recovered inside the tree by using the fact that for $\beta<\gamma<\alpha$, the initial segment of $c_A$ of length $\beta$ and the corresponding segment for the set $\{\gamma\}$ determine the order relation.

\begin{theorem}
    For every ordinal $\alpha$, every $\mathcal L_{n,S}$-formula over $\mathfrak T_\alpha^n$ can be translated into an MSO formula over $(\alpha,<)$. Conversely, every MSO formula over $(\alpha,<)$ can be translated into an $\mathcal L_{n,S}\cup\{\sim\}$-formula over $(\mathfrak T_{\alpha+1}^n,\sim)$.
\end{theorem}
\begin{proof}
    For the first translation, use the partition code above; the tree order and the functions $S_i$ are MSO-definable on these codes. Conversely, code $A\subseteq\alpha$ by $c_A$ and $\beta<\alpha$ by $c_{\{\beta\}}$. On these codes, the relations $\beta<\gamma$ and $\beta\in A$ are first-order definable from $\subseteq$, $S_0$, $S_1$, and $\sim$. Induction on formulae gives both translations.
\end{proof}

 In contrast to the interpretation of the previous subsection, this one is not uniform in the ordinal, as it requires a shift by $+1$ to obtain the full power set; but for fixed $\alpha$ it gives a precise calibration of the expressive power of the tree with equal height.

The interpretations above make essential use of the equal‑height relation. We now show that this relation is \emph{not} definable in the pure first‑order language of $\mathfrak{T}_\alpha^n$.

\begin{theorem}[Undefinability in theories]\label{thm:height-not-definable-theory}
    Let $\Sigma$ be the elementary theory of any standard tree $\mathfrak{T}_\alpha^n$, and let $\varphi(x,y)$ be an $\mathcal{L}_{n,S}$-formula.  Then $\varphi$ does not define the relation $H(x,y)$ which holds iff $\operatorname{ht}(x)=\operatorname{ht}(y)$ in all models of $\Sigma$.
\end{theorem}
\begin{proof}
  Every standard tree is elementarily equivalent to one of the canonical trees analysed in Section~\ref{sec:type1syn}--\ref{sec:type34syn}, and for each canonical tree we have established quantifier elimination down to a suitable diagram language (Theorems~\ref{thm:1-qe}, \ref{thm:2-qe}, \ref{thm:3-qe}).
  For every finite quantifier rank $q$, there exist nodes $u,v,t$ such that $\operatorname{ht}(u)=\operatorname{ht}(t)\neq\operatorname{ht}(v)$ while the pairs $(u,t)$ and $(u,v)$ realize the same rank-$(q+2)$ diagram.

\end{proof}

\begin{theorem}[Undefinability in standard structures]\label{thm:height-not-definable-structure}
    For every $\alpha\ge\omega$, the relation $H(x,y)$ which holds iff $\operatorname{ht}(x)=\operatorname{ht}(y)$ is not definable by a first‑order formula in the structure $\mathfrak{T}_\alpha^n$.
\end{theorem}
\begin{proof}

  Fix a quantifier rank $q$. Choose finite ordinals $2^{q+2}-1\leq\beta<\gamma<\omega$. There are colored words of lengths $\beta$ and $\gamma$ with the same rank-$(q+2)$ characteristic. Choose nodes $u,t$ of height $\beta$ and a node $v$ of height $\gamma$ so that $u$ lies in one immediate-successor cone and $t,v$ lie in another, with the transition words to $t$ and $v$ having that common characteristic. The pairs $(u,t)$ and $(u,v)$ then realize the same rank-$(q+2)$ diagram and satisfy the same formulas of rank at most $q$, although $\operatorname{ht}(u)=\operatorname{ht}(t)\neq\operatorname{ht}(v)$. Therefore no first-order formula defines equality of height.

\end{proof}

Together, these results clarify the tight connection between full $n$-branching ordinal trees and monadic second‑order logic, while confirming that the natural metric structure given by height is inherently second‑order in character.

\printbibliography

\end{document}